\documentclass[12pt]{amsart}

\usepackage{latexsym}
\usepackage{amssymb}
\usepackage[cp850]{inputenc}
\usepackage{epsfig}
\usepackage{psfrag}
\usepackage{amsthm}
\usepackage{amscd}
\usepackage{amsmath}
\usepackage{amsfonts}
\usepackage{graphics}
\usepackage[all]{xy}

\theoremstyle{plain}
\newtheorem{sat}{Theorem}[section]

\newtheorem{lem}[sat]{Lemma}
\newtheorem{kor}[sat]{Corollary}
\newtheorem{prop}[sat]{Proposition}

\newtheorem*{defi*}{Definition}
\newtheorem*{bei*}{Example}
\newtheorem*{sat*}{Theorem}
\newtheorem*{kor*}{Corollary}
\newtheorem*{rmk*}{Remark}
\newtheorem{quest}{Question}

\newtheorem*{namedtheorem}{\theoremname}
\newcommand{\theoremname}{testing}
\newenvironment{named}[1]{\renewcommand{\theoremname}{#1}\begin{namedtheorem}}{\end{namedtheorem}}

\theoremstyle{definition}

\let\ssection=\section
\renewcommand{\section}{\setcounter{equation}{0}\ssection}

\renewcommand{\setminus}{{\smallsetminus}}

\theoremstyle{remark}
\newtheorem*{bem}{Remark}

\newcommand{\BC}{\mathbb C}
\newcommand{\BH}{\mathbb H}
\newcommand{\BR}{\mathbb R}
\newcommand{\BD}{\mathbb D}

\newcommand{\BS}{\mathbb S}
\newcommand{\BZ}{\mathbb Z}

\newcommand{\BT}{\mathbb T}

\newcommand{\CC}{\mathcal C}

		\newcommand{\CL}{\mathcal L}
\newcommand{\CM}{\mathcal M}		\newcommand{\CN}{\mathcal N}
		\newcommand{\CP}{\mathcal P}

\newcommand{\D}{\partial}

\def\co{\colon\thinspace}

\DeclareMathOperator{\PSL}{PSL}		%	Spezielle lineare Gruppe
\DeclareMathOperator{\vol}{vol}		%	Volumen

\DeclareMathOperator{\inj}{inj}

\begin{document}

\title{Geometric limits of knot complements}
\author{Jessica S. Purcell}
\author{Juan Souto}
\begin{abstract}
We prove that any complete hyperbolic 3--manifold with finitely
generated fundamental group, with a single topological end, and which
embeds into $\BS^3$ is the geometric limit of a sequence of hyperbolic
knot complements in $\BS^3$. In particular, we derive the existence of
hyperbolic knot complements which contain balls of arbitrarily large
radius. We also show that a complete hyperbolic 3--manifold with two
convex cocompact ends cannot be a geometric limit of knot complements
in $\BS^3$.
\end{abstract}
\maketitle

%%%%%%%%%%%%%%%%%%%%%%%%%%%%%%%%%%%%%%%%%%%%%%%%%%%%%%%%%%%%%%%%%
\section{Introduction}

In this paper we study geometric properties of hyperbolic knot
complements.  Unless explicitly stated, a \emph{knot complement} is
understood to be the complement of a non-trivial knot in the 3--sphere
$\BS^3$.  A non-trivial knot $K\subset \BS^3$ is \emph{hyperbolic} if
its complement $M_K$ admits a complete hyperbolic metric. Many knots
are hyperbolic; in fact, among the 1.701.936 prime knots with 16 or
fewer crossings, all but 32 are hyperbolic \cite{Hoste}. In general,
Thurston proved that every knot $K \subset \BS^3$ which is neither a
torus nor a satellite knot is hyperbolic.

By the Mostow--Prasad theorem, the hyperbolic metric on a hyperbolic
knot complement is unique.  In particular it follows from work of
Gordon and Luecke \cite{gordon-luecke} that the hyperbolic metric on
the complement of a hyperbolic knot in $\BS^3$ is a complete invariant
of the knot. However, it remains a difficult problem to use geometric
arguments in the classification of knots.  One central difficulty is
in distinguishing hyperbolic knot complements from general finite
volume hyperbolic manifolds.

\begin{quest}
Which properties of the hyperbolic metric distinguish knot complements
from other hyperbolic 3--manifolds?
\end{quest}

There are some results suggesting that within the world of finite
volume hyperbolic 3--manifolds, hyperbolic knot complements form a
very special class.  For instance, Reid \cite{Reid-knot} proved that
the complement of the figure-8 knot is the only hyperbolic arithmetic
knot complement.  Similarly, every hyperbolic knot complement admits
infinitely many non-isometric finite degree covers, but at most three
of them are knot complements \cite{gaw}.

In this paper, we investigate the question above by studying geometric
limits of hyperbolic knot complements.  We find that many manifolds
arise as these geometric limits, which implies that knot complements
can admit metrics with unusual and perhaps unexpected geometric
properties.  We also investigate manifolds that cannot arise in this
manner.

Recall that a hyperbolic manifold $M$ is a geometric limit of a
sequence of hyperbolic manifolds $M_i$, if there are basepoints $p\in
M$ and $p_i \in M_i$ such that larger and larger balls about $p$ in
$M$ have, as $i$ tends to $\infty$, better and better almost isometric
embeddings into balls about $p_i$ in $M_i$.  Thus if a manifold is a
geometric limit of knot complements, then there exist hyperbolic knots
whose geometric properties are very close to those of the limiting
manifold.

Our first result asserts that surprisingly many hyperbolic
3--manifolds arise as geometric limits of knot complements.
Specifically, we prove:

\begin{sat}\label{thm:main}
Let $N$ be a complete hyperbolic 3--manifold with finitely generated
fundamental group and a single topological end.  If $N$ is
homeomorphic to a submanifold of $\BS^3$, then $N$ is a geometric
limit of a sequence of hyperbolic knot complements in $\BS^3$.
\end{sat}

%% Recall that a hyperbolic manifold $M$ is a geometric limit of a
%% sequence $(M_i)$ if there are base points $p\in M$ and $p_i\in M_i$
%% such that for all $r$ and $\epsilon$ there is $i_{\epsilon,r}$ such
%% that for all $i \ge i_{\epsilon,r}$, the ball $B_M(p,R)$ in $M$ with
%% center $p$ and radius $R$ can be mapped into $M_i$, mapping $p$ to
%% $p_i$, by a map which is $(1+\epsilon)$-bi-Lipschitz onto its image.

Before moving on, observe that applying Theorem \ref{thm:main} to
$N=\BH^3$ we obtain:

\begin{kor}\label{large-inj}
For every $R>0$ there exists a hyperbolic knot complement $M_K$ and
$x\in M_K$ with injectivity radius $\inj(x,M_K)>R$.
\end{kor}

The statement of Theorem \ref{thm:main} is not optimal.  Namely, we
can construct many other geometric limits of knot complements by
observing that every hyperbolic 3--manifold which is a geometric limit
of manifolds satisfying the conditions in Theorem \ref{thm:main} is
also a geometric limit of hyperbolic knot complements.  For example,
this idea can be used to prove that every hyperbolic manifold $N$
homeomorphic to the trivial interval bundle over a closed surface
which has at least one degenerate end is also a limit of knot
complements.  Again the same reasoning shows that there are hyperbolic
3--manifolds with finitely generated fundamental group and arbitrarily
many ends which are geometric limits of knot complements.  We discuss
these in section \ref{sec:various}.

We then turn to the converse problem, and ask which hyperbolic
3--manifolds cannot be the geometric limit of hyperbolic knot
complements.  Besides obvious topological obstructions, we find there
are some less obvious geometric obstructions.  In particular, we find:

\begin{sat}\label{yair}
Let $M$ be a hyperbolic 3--manifold. If the manifold $M$ has at least
two convex cocompact ends, then $M$ is not the geometric limit of any
sequence of hyperbolic knot complements in $\BS^3$.
\end{sat}

Recall that if $M$ has two convex cocompact ends, then every regular
neigbhorhood of the convex core $CC(M)$ in $M$ has at least two
compact boundary components, where $CC(M)$ is the smallest closed
totally convex subset of $M$.

%% Recall that the convex core $CC(M)$ of a hyperbolic 3-manifold $M$ is
%% the smallest closed totally convex subset of $M$. The assumption that
%% $M$ has at least two convex cocompact ends means that some, and hence
%% every, regular neighborhood of $CC(M)$ in $M$ has at least two compact
%% boundary components.

%Theorem \ref{yair} has a surprising consequence.  Namely, if $\Gamma
%\subset \PSL_2 \BR$ is a discrete, torsion--free, cocompact Fuchsian
%group then the manifold $\BH^3/\Gamma$ is convex cocompact and has
%two ends.  Hence, by Theorem \ref{yair}, $\BH^3/\Gamma$ is not a
%geometric limit of knot complements.  On the other hand, if
%$\BH^2/\Gamma$ is not cocompact, then we deduce from Theorem
%\ref{thm:main} that $\BH^3/\Gamma$ is indeed a geometric limit of knot
%complements.  Combining these two observations we obtain that a Fuchsian hyperbolic 3-manifold 
%$\BH^3/\Gamma$ is a geometric limit of knot complements in $\BS^3$ if
%and only if the surface $\BH^2/\Gamma$ is open.

\medskip
\noindent {\bf Organization and overview.} The following gives a road
map for this paper.  In section \ref{sec:prelim}, we remind the reader
of basic and not so basic facts on hyperbolic 3--manifolds and their
deformation theory.

The main result of section \ref{sec:kleinian}, namely Proposition
\ref{prop:only-degenerate}, is that any manifold satisfying the
hypotheses of Theorem \ref{thm:main} is the geometric limit of
hyperbolic manifolds which, besides satisfying the same hypotheses,
have the property that their single topological end is degenerate.  We
also prove Proposition \ref{prop:limit-degenerate}: that if $M$ is a
hyperbolic 3--manifold with a degenerate end, and if
$\{M_{\gamma_i}\}$ is any sequence of hyperbolic 3--manifolds
homeomorphic to the interior of $M$ with parabolics $\gamma_i$
converging to the ending lamination of $M$, then $M$ is a geometric
limit of $\{M_{\gamma_i}\}$.

Thus at this step, it remains to show that there are geometric limits
of knot complements homeomorphic to $M$ and where certain curves are
parabolic.  This is done in Proposition \ref{prop:reduction}, in
section \ref{sec:knots}.  We start constructing the desired manifold
as a limit of hyperbolic link complements.  In order to go from link
complements to knot complements, we perform hyperbolic Dehn filling on
certain components.  The links are chosen in such a way that the
slopes of Dehn filling get long.  We derive from a version of Hodgson
and Kerckhoff's quantified Dehn filling theorem that the filled
manifolds have the same geometric limit.

Theorem \ref{thm:main} follows easily from Proposition
\ref{prop:only-degenerate}, Proposition \ref{prop:reduction} and
Proposition \ref{prop:limit-degenerate}.  For the convenience of the
reader, we invert the logical order, deriving Theorem \ref{thm:main}
from these results in section \ref{sec:proof}.

In section \ref{sec:yair} we prove Theorem \ref{yair}.  The proof is
based on the fact that the closure of the complement of the
convex--hull of a subset of $\D_\infty\BH^3$ is a locally
CAT(-1)--manifold.  This fact allows us to adapt an argument due to
Marc Lackenby \cite{Lackenby-attaching}.

Finally, in section  \ref{sec:various} we discuss briefly
Corollary \ref{large-inj} and some other
consequences of Theorem \ref{thm:main} and Theorem \ref{yair}.  We also
propose some questions.
\medskip

\noindent{\bf Acknowledgements.}
The second author would like to thank the Mathematics Institute of
Oxford University for its hospitality while this work was
conceived. We would like to thank Larry Guth, Marc Lackenby and Yair
Minsky for very helpful conversations.

The first author has been partially supported by the Leverhulme trust,
and by NSF grant DMS-0704359. The second author has been partially
supported by NSF grant DMS-0706878 and the Alfred P. Sloan Foundation.

%%%%%%%%%%%%%%%%%%%%%%%%%%%%%%%%%%%%%%%%%%%%%%%%%%%%%%%%%%%%%%%%%

\section{Hyperbolic 3--manifolds and geometric limits}\label{sec:prelim}

Throughout this note we will only consider knot complements in
$\BS^3$.

%%%%%%%%%%%%%%%%%%%%%%%%%%%%%%%%%%%%%%%%%%%%%%%%%%%%%%%%%%%%%%%%%

Under a \emph{hyperbolic 3--manifold} we understand a connected and
orientable, complete Riemannian manifold $M$ of constant curvature
$-1$.  Equivalently, $M$ is isometric to $\BH^3/\Gamma$, where
$\Gamma$ is a discrete, torsion free subgroup of the group of
isometries of hyperbolic 3--space $\BH^3$.

By Mostow--Prasad rigidity, any two homotopy equivalent, finite volume
hyperbolic 3--manifolds are isometric.  On the other hand, if the
volume of $M$ is infinite, then there is a rich and well--developed
deformation theory.

In this section, we review a few aspects of this deformation theory.
We refer to \cite{Japaner} for details of facts which are, by now,
classical, and to \cite{Marden} for more recent results.

\subsection{Tameness}
Given a hyperbolic 3--manifold $M$ with finitely generated fundamental
group, there is an associated compact 3--manifold with boundary
$\overline{M}$, the {\em manifold compactification of $M$}, such that
$M$ is homeomorphic to the interior of $\overline{M}$.  This is the
result of the \emph{tameness theorem}, proved by Agol \cite{Agol}, and
Calegari and Gabai \cite{Calegari-Gabai}. It is well known that
$\overline M$ is unique up to homeomorphism, that every component of
$\D\overline M$ has at least genus $1$ and that $\chi(\D\overline
M)=0$ if and only if $M$ has finite volume.

Recall that a non-trivial element $\gamma \in \pi_1(M)$ is
\emph{parabolic} if it is freely homotopic to curves in $M$ of
arbitrarily short length.  A subsurface $P \subset \D\overline{M}$ is
the \emph{parabolic locus of $M$} provided it satisfies:
\begin{enumerate}
\item $P$ consists of all toroidal components of $\D\overline{M}$ and
a collection of homotopically essential, disjoint, non-parallel
annuli.
\item If $\gamma$ is a homotopically non-trivial curve in $P$, then
any element in $\pi_1(M)$ represented by $\gamma$ is parabolic.
\item If $\gamma \subset M$ represents a parabolic element and
$\gamma_i$ is a sequence of curves freely homotopic to $\gamma$ and
with $\ell_M(\gamma_i) \to 0$, then for every regular neighborhood $U$
of $P$ in $\overline M$, there is $i_U$ with $\gamma_i \subset U$ for
all $i \ge i_U$.
\end{enumerate}
If $\overline M$ is the manifold compactification of the hyperbolic
3--manifold $M$, then any two incompressible subsurfaces $P_1, P_2
\subset \D\overline M$ satisfying conditions (1)--(3) are isotopic to
each other within $\D\overline M$. 

In other words, given a hyperbolic 3--manifold $M$ with finitely
generated fundamental group, $(\overline M, P)$ is uniquely
determined up to homeomorphism of pairs.  In particular, given a
hyperbolic 3--manifold $M$, we will abuse terminology only slightly
when we use the definite article in the sentence: \emph{$\overline M$
is the manifold compactification of $M$ and $P \subset \D\overline M$
is the parabolic locus.}

%In particular, given a hyperbolic
%3--manifold $M$ with finitely generated fundamental group, we may
%refer to $(\overline{M}, P)$ without confusion.

\begin{defi*}
Let $M$ be a hyperbolic 3--manifold with finitely generated
fundamental group; let $\overline M$ be its manifold compactification
and $P \subset \D\overline M$ its parabolic locus.  The components of
$\D\overline{M} \setminus P$ are the \emph{geometric ends} of $M$.
\end{defi*}

Observe that whenever the parabolic locus of $M$ is empty, then
geometric and topological ends coincide.

\subsection{Convex core}
Let $M = \BH^3/\Gamma$ be a hyperbolic 3--manifold and recall that we
can identify the boundary at infinity of $\BH^3$ with the complex
projective line $\D_\infty\BH^3 = \BC P^1$. Denote by $\Lambda_\Gamma$
the {\em limit set} of $\Gamma$ and let $CH(\Lambda_\Gamma) \subset
\BH^3$ be its convex hull.  The quotient $CC(M) = CH(\Lambda_\Gamma) /
\Gamma$ is the {\em convex core of $M$}.  Equivalently, $CC(M)$ is the
smallest closed, totally convex subset of $M$, i.e. the smallest
closed subset which contains all closed geodesics.  The manifold $M$
is \emph{convex cocompact} if its convex core is compact.  Observe
that if $M$ is convex cocompact then its parabolic locus is empty.

More generally, a hyperbolic 3--manifold $M$ with finitely generated
fundamental group is {\em geometrically finite} if its convex core has
finite volume $\vol(CC(M)) < \infty$.  One important class of
geometrically finite manifolds are the maximal cusps.  A hyperbolic
3--manifold $M$ is said to be a \emph{maximal cusp} if $(\overline{M},
P)$ is such that every component of $\D \overline{M} \setminus P$ is
homeomorphic to a 3--punctured sphere.  Note that if $M$ is a maximal
cusp, then its convex core $CC(M)$ must have totally geodesic
boundary, with each boundary component isometric to the unique
hyperbolic structure on a 3--punctured sphere.

\subsection{Degenerate ends}
Throughout the paper, we will be mostly interested in hyperbolic
3--manifolds which have a single geometric end.  To avoid unnecessary
notation and terminology, the following definition is tailored to the
particular situation of this paper:

\begin{defi*}
Let $M$ be a hyperbolic 3--manifold with finitely generated
fundamental group and a single geometric end. The geometric end of $M$
is said to be \emph{degenerate} if $M$ does not contain any proper
totally convex subset, or equivalently, if $M=CC(M)$.
\end{defi*}

Intuitively, manifolds with degenerate ends are more complicated than
convex cocompact manifolds.  However, the former are much more rigid
than the latter; this is going to be crucial in this paper.  We find
the first manifestation of the rigidity of degenerate ends in the
following result, due to Thurston \cite{thurston} and Canary
\cite{Canary-covering}, which we state in the situation we are
interested in:

\begin{named}{Covering theorem}
Assume that $M$ is a hyperbolic 3--manifold with finitely generated
fundamental group and with a unique geometric end. Assume that the
geometric end of $M$ is degenerate and that $\pi \co M\to N$ is a
Riemannian cover where $N$ has infinite volume.  Then the cover $\pi$
is finite--to--one.
\end{named}

Suppose now that $M$ has a single geometric end which is degenerate.
Then we have an associated \emph{ending lamination} $\lambda$, whose
definition we recall.

We may identify $\overline{M}$ with a submanifold of $M$ such that the
complement $M \setminus \overline{M}$ has a product structure.
Observe that any curve $\gamma \subset M\setminus \overline{M}$
corresponds to a unique homotopy class of curves in $\D\overline{M}$.
Now if the end of $M$ is degenerate, then there is a sequence of
closed geodesics $\gamma_i$ with bounded length, exiting the end,
represented by simple closed curves in $\D\overline{M} \setminus P$.
A subsequence of the curves $\gamma_i$ converges in
$\mathcal{PML}(\D\overline{M} \setminus P)$ to a filling measured
lamination whose support, $\lambda$, does not depend on the sequence
$\gamma_i$.
%% (Recall a projective measured lamination $\lambda$ is \emph{filling}
%% if $i(\lambda, \gamma)>0$ for every essential simple closed curve
%% $\gamma$, where $i(\cdot, \cdot)$ is the geometric intersection
%% number.) 
The lamination $\lambda$ is the \emph{ending lamination} of
$M$.  See Thurston \cite{thurston}, Bonahon \cite{Bonahon86}, or
Canary \cite{Canary-ends} for basic properties of the ending
lamination.

The ending lamination theorem, due to Minsky \cite{ELC1} and Brock,
Canary, and Minsky \cite{ELC2}, asserts that every hyperbolic
3--manifold is determined up to isometry by its topology and end
invariants.  We state this theorem again in the particular case we are
interested in.

\begin{named}{Ending lamination theorem}
Let $\overline{M}$ be a compact 3--manifold with boundary, and $P
\subset \D\overline{M}$ a possibly empty subsurface with
$\D\overline{M} \setminus P$ connected.  Assume that $M$ and $M'$ are
hyperbolic 3--manifolds homeomorphic to the interior of $\overline{M}$
with parabolic locus $P$.  Assume also that the geometric ends of $M$
and $M'$ are degenerate and have the same ending lamination.  Then $M$
and $M'$ are isometric.
\end{named}

\subsection{Doubly incompressible laminations}
Before going any further recall the following definition:

\begin{defi*}
Let $\overline M$ be a compact 3--manifold whose interior admits a
hyperbolic structure.  Let $P \subset \D\overline M$ be a possibly
empty subsurface consisting of all toroidal components of $\D\overline
M$ and a collection of disjoint, non-parallel, homotopically essential
annuli.  We say that $(\overline{M}, P)$ is \emph{acylindrical} if
\begin{enumerate}
	\item every component of $\D\overline{M} \setminus P$ is
	incompressible in $\overline{M}$,
	\item every properly embedded annulus $(A, \D A) \subset
	(\overline{M}, \D\overline{M} \setminus P)$ is isotopic relative to
	the boundary to an annulus contained in the boundary
	$\D\overline{M}$, and
	\item there is no properly embedded M\"obius band with boundary in
	$\D\overline{M} \setminus P$.
\end{enumerate}
\end{defi*}

A generalization of this notion is due to Kim, Lecuire, and Ohshika
\cite{KLO}, who defined a measured lamination $\alpha \in
\CP\CM\CL(\D\overline M \setminus P)$ to be \emph{doubly
incompressible}, if there is $\eta>0$ such that $i(\alpha, \D E) >
\eta$ for $E$ any essential annulus, M\"obius band, or disc.  Observe
that if $\gamma \subset \D\overline M \setminus P$ is doubly
incompressible when considered as an element in $\CP\CM\CL(\D\overline
M \setminus P)$, then $(\overline M, \CN(\gamma) \cup P)$ is
acylindrical; here $\CN(\gamma)$ is a regular neighborhood of $\gamma$
in $\D\overline M$.

The following result, which we state only in the setting we are
interested in, asserts essentially that ending laminations are doubly
incompressible.

\begin{sat}[Canary]\label{Canary-masur}
Assume that a hyperbolic 3--manifold $M$ with finitely generated
fundamental group has a single geometric end and that this end is
degenerate with ending lamination $\lambda$.  If $\alpha$ is any
measured lamination with support $\lambda$, then $\alpha$ is doubly
incompressible.

Moreover, if $\gamma_i$ is any sequence of simple closed curves in
$\D\overline M \setminus P$ converging in $\CP\CM\CL(\D\overline M
\setminus P)$ to $\alpha$, then $(\overline M, \CN(\gamma_i) \cup P)$
is acylindrical for all sufficiently large $i$.  Here, $\CN(\gamma_i)$
is a regular neighborhood of $\gamma_i$ in $\D\overline M$.
\end{sat}

\begin{bem}
If $\D\overline M \setminus P$ is incompressible, then Theorem
\ref{Canary-masur} follows from the work of Thurston \cite{thurston}.
Canary proved the first claim of Theorem \ref{Canary-masur} if
$\D\overline M$ is compressible and the parabolic locus is empty.  In
fact, in this case he showed that $\lambda$ belongs to the so--called
Masur domain, a subset of the set of doubly incompressible laminations
\cite{Canary-ends}.  Since the Masur domain is open \cite{Masur86,
Otal88}, the second claim follows.  Canary's argument goes through
without any problems in the presence of parabolics.
\end{bem}

\subsection{Pleated surfaces and non-realized laminations}
Assume that $M$ is a hyperbolic 3--manifold with manifold
compactification $\overline M$ and parabolic locus $P$, and let
$\overline S$ be a compact surface with interior $S = \overline S
\setminus \D\overline S$.  A {\em pleated surface} is a map
$$\phi \co S \to M$$ such that there is a finite volume hyperbolic
metric $\sigma$ on $S$ such that the following holds:
\begin{enumerate}
\item The image under $\phi$ of a boundary parallel curve in $S$
represents a parabolic element in $\pi_1(M)$.
\item The map $\phi \co (S,\sigma) \to M$ preserves the lengths of
paths.
\item Every point in $x$ is contained in an arc $\kappa$, geodesic
with respect to $\sigma$, such that the restriction of $\phi$ to
$\kappa$ is an isometric embedding.
\end{enumerate}
See for instance \cite{CEG} for basic properties of pleated surfaces. 

If we fix a homotopy class of maps $[S\to M]$ then we say that a
lamination $\lambda \subset S$ is {\em realized} in $M$ by a pleated
surface $\phi \co S\to M$ in the correct homotopy class if:
\begin{itemize}
\item[(4)] The restriction of $\phi$ to each leaf of $\lambda$ is an
isometric immersion.
\end{itemize}

The following is a very technical result which essentially asserts
that if a filling lamination is not realized, then it is the ending
lamination.

\begin{prop}[Non-realized implies ending lamination]\label{detect-ending}
Assume that $\overline M$ is a compact 3--manifold and $P \subset
\D\overline M$ is a subsurface such that there is some hyperbolic
manifold with manifold compactification $\overline M$ and parabolic
locus $P$.  Assume moreover that $S = \D\overline M \setminus P$ is
connected.

Let $N$ be a hyperbolic 3--manifold and $f \co \overline M \to N$ a
homotopy equivalence mapping each curve in $P$ to a parabolic element
in $\pi_1(N)$.  Assume that that there is a filling doubly
incompressible lamination $\lambda \subset S = \D\overline M
\setminus P$ which is not realized by any pleated surface homotopic
to the restriction of $f$ to $S$.  Then the following holds:
\begin{itemize}
\item $\overline M$ is the manifold compactification of $N$ and $P$
its parabolic locus,
\item the only geometric end of $N$ is degenerate, and has ending
lamination $\lambda$.
\end{itemize}
\end{prop}

Proposition \ref{detect-ending} is due to Thurston \cite{thurston} if
$\pi_1(M)$ does not split as a free product.  If $\pi_1(M)$ splits as
a free product but is not free, then Proposition \ref{detect-ending}
is due to Kleineidam and Souto \cite{Kleineidam-Souto}.  The case that
$\pi_1(M)$ is free has been treated by Namazi and Souto
\cite{Namazi-Souto} and Ohshika.

\begin{bem}
Recall that the manifold compactification and the parabolic locus of a
hyperbolic 3--manifold are only determined up to homeomorphism.
Therefore, to be slightly more precise, the statement of Proposition
\ref{detect-ending} should be that there is a homeomorphism $F \co N
\to \overline M \setminus \D\overline M$ that satisfies properties of
the Proposition.  
\end{bem}

\subsection{Geometric limits}
A sequence $(M_i, p_i)$ of pointed hyperbolic 3--manifolds
\emph{converges geometrically}, or equivalently, converges in the
\emph{pointed Gromov--Hausdorff topology}, to a pointed manifold $(M,
p)$ if for every $\epsilon > 0$ and every compact set $K \subset M$
with $p \in K$, there exists $i_{\epsilon, K}$ such that for all $i
\ge i_{\epsilon,K}$ there is a $(1+\epsilon)$--bilipschitz embedding
$$f_i\co (K,p) \hookrightarrow (M_i, p_i).$$

Observe that if $(M_i, p_i)$ converges geometrically to $(M, p)$ and
$q \in M$ is a second base point, then there are points $q_i \in M_i$
such that $(M_i, q_i)$ converges geometrically to $(M, q)$.  On the
other hand, it is not difficult to construct sequences of manifolds
$(M_i)$ and two sequences of base points $p_i, q_i \in M_i$ such that
the $(M_i, p_i)$ and $(M_i, q_i)$ converge geometrically to
non-homeomorphic limits.  This last remark explains the undetermined
article in the following definition:

\begin{defi*}
A (connected) hyperbolic 3--manifold $M$ is a \emph{geometric limit}
of a sequence of hyperbolic manifolds $(M_i)$ if there are base points
$p \in M$ and $p_i \in M_i$ such that $(M_i, p_i)$ converges
geometrically to $(M, p)$.
\end{defi*}

The following is an obvious but extremely useful lemma.

\begin{lem}\label{lemma:silly}
Assume that a manifold $M$ is a geometric limit of a sequence of
hyperbolic 3--manifolds $M_i$ and that each $M_i$ is a geometric limit
of a sequence of hyperbolic knot complements.  Then $M$ is also a
geometric limit of a sequence of hyperbolic knot complements.\qed
\end{lem}

\subsection{Algebraic limits}
We recall now a second concept of convergence.

Let $M$ be a hyperbolic 3--manifold.  Let $AH(M)$ denote the space of
conjugacy classes of discrete faithful representations of $\pi_1(M)$
into $\PSL_2\BC$.  Give $AH(M)$ the quotient topology induced by the
compact--open topology on the space of discrete faithful
representations.  Convergence of representations in $AH(M)$ is called
\emph{algebraic convergence}.  If $\{\rho_n\}$ converges algebraically
to $\rho$, then the manifold $\BH^3/ \rho(\pi_1(M) )$ is the
\emph{algebraic limit} of the 3--manifolds $\BH^3/ \rho_n(\pi_1(M))$.
See \cite{McMullen} or \cite{biringer-souto} for a description of
algebraic convergence from the point of view of the involved
hyperbolic manifolds.

\begin{named}{Density theorem}
Assume that $\BZ^2 \not\leq \pi_1(M)$.  Then the set of those
$\rho\in AH(M)$ such that the associated manifold $\BH^3 /
\rho(\pi_1(M))$ is convex cocompact is dense in $AH(M)$.
\end{named}

The density theorem follows from the proof of the tameness conjecture
by Agol, Calegari and Gabai, the proof of the ending lamination
theorem by Brock, Canary, and Minsky, and the work of Kim, Lecuire and
Ohshika.  The final step needs the more general form of Proposition
\ref{detect-ending} above due to Namazi, Ohshika and Souto.  See
\cite{Marden} for the relation between all these results.
\medskip

Before moving on we state a very weak form of the continuity of
Thurston's length function (see Brock \cite{Brock-length} for an extensive
discussion of the length function).

\begin{prop}\label{realized-continuity}
Assume that $\overline M$ is a compact 3--manifold and $P \subset
\D\overline M$ is a subsurface such that there is some hyperbolic
manifold with manifold compactification $\overline M$ and parabolic
locus $P$.  Assume moreover that $S = \D\overline M \setminus P$ is
connected and let $\lambda\subset S$ be a filling doubly
incompressible lamination.

Let $\{M_i\}$ be a sequence of hyperbolic 3--manifolds in
$AH(\overline M)$ such that every curve in $P$ is parabolic in $M_i$
for all $i$ and assume that the sequence $\{M_i\}$ converges
algebraically to some manifold $M$ in which $\lambda$ is realized.
Then we have
$$\lim_i l_{M_i}(\gamma_i)=\infty$$ for every sequence of simple
closed curves $\{\gamma_i\}$ in $S$ converging to $\lambda$ in
$\CP\CM\CL(S)$.  Here $l_{M_i}(\gamma_i)$ is the infimum of the
lengths in $M_i$ of all curves freely homotopic to $\gamma_i$.
\end{prop}

\subsection{Strong limits}
Algebraic limits are not necessarily the same as geometric limits, see
for example \cite[Chapter 9]{thurston}.

\begin{defi*}
Suppose $\{M_n\}$ is a sequence of hyperbolic 3--manifolds that
converges algebraically and geometrically to the manifold $M$.  Then
we say $\{M_n\}$ \emph{converges strongly} to $M$.
\end{defi*}

In the course of the proof of Theorem \ref{thm:main} we will need to
be able to deduce that some algebraically convergent sequences also
converge strongly.  Our main tool is the following result which follows
easily from the Canary--Thurston covering theorem.

\begin{prop}\label{alg-strong}
Assume that a sequence $\{\rho_i\}$ in $AH(M)$ converges algebraically
to some $\rho \in AH(M)$ and suppose that $CC(\BH^3 / \rho(\pi_1(M)))
= \BH^3/\rho(\pi_1(M))$.  Then the sequence $\{\rho_i\}$ converges
strongly to $\rho$.
\end{prop}

See \cite{AC1,AC2,Gero} for related, much more powerful results.

Proposition \ref{alg-strong} does not apply if $\BH^3/\rho(\pi_1(M))$
is a maximal cusp.  However, in this case, we have the following
weaker result which follows directly from \cite[Prop. 3.2]{accs}:

\begin{prop}\label{prop:accs}
Assume that a sequence $\{\rho_i\}$ in $AH(M)$ converges algebraically
to some $\rho\in AH(M)$ and suppose that $\BH^3/\rho(\pi_1(M))$ is a
maximal cusp.  Then, up to passing to a subsequence, the hyperbolic
manifolds $\BH^3/\rho_i(\pi_1(M))$ converge geometrically to some $N$
such that there is an isometric embedding $CC(\BH^3/\rho(\pi_1(M)))
\hookrightarrow N$.
\end{prop}

Before going on to more interesting topics, we recall the following
strong density theorem:

\begin{named}{Strong density theorem}
Assume that $\BZ^2 \not\leq \pi_1(M)$.  Then any $\rho\in AH(M)$ is
the strong limit of a sequence $\{\rho_i\}$ in $AH(M)$ such that for
each $i$ the associated manifold $\BH^3/\rho_i(\pi_1(M))$ is convex
cocompact.
\end{named}

The strong density theorem follows directly from the Density theorem
and work of Brock and Souto \cite[Theorem 1.4] {Brock-Souto}.

%%%%%%%%%%%%%%%%%%%%%%%%%%%%%%%%%%%%%%%%%%%%%%%%%%%%%%%%%%%%%%%%%

\section{Proof of the main theorem}\label{sec:proof}
In this section we reduce the proof of Theorem \ref{thm:main} to
various results obtained in the two subsequent sections.

\begin{named}{Theorem \ref{thm:main}}
Let $N$ be a complete hyperbolic 3--manifold with finitely generated
fundamental group and a single topological end.  If $N$ is
homeomorphic to a submanifold of $\BS^3$, then $N$ is a geometric
limit of a sequence of hyperbolic knot complements.
\end{named}

\begin{proof}

Assume that $N$ is as in the statement of Theorem \ref{thm:main} and
let $\overline N$ be its manifold compactification.

First, we note that we need not consider trivial cases.  If $N$ has
abelian fundamental group, then using for instance Klein combination,
one obtains a sequence $\{N_i\}$ of say genus 2 handlebodies
converging geometrically to $N$.  By Lemma \ref{lemma:silly}, it
suffices to prove that each one of the $N_i$ is a geometric limit of
hyperbolic knot complements.  From now on we assume that $\pi_1(N)$ is
not abelian.  Similarly, if $\D\overline N$ is a torus, then the fact
that $N$ embeds into $\BS^3$ implies that $N$ is itself a knot
complement.  So we may assume that $\D\overline N$ is a connected
surface of genus at least two.  In other words, we may assume than $N$
has infinite volume.

In section \ref{sec:kleinian}, we will prove:

\begin{named}{Proposition \ref{prop:only-degenerate}}
Let $N$ be a complete, infinite volume, hyperbolic 3--manifold with
non-abelian, finitely generated fundamental group and a single
topological end.  Assume that $N$ is homeomorphic to a submanifold of
$\BS^3$.  Then $N$ is a geometric limit of a sequence $\{N_i\}$ of
hyperbolic 3--manifolds such that for all $i$ the following holds:
\begin{itemize}
\item $N_i$ has finitely generated fundamental group, no parabolics
and a single end,
\item the end of $N_i$ is degenerate, and
\item $N_i$ admits an embedding into $\BS^3$.
\end{itemize}
\end{named}

Combining Proposition \ref{prop:only-degenerate} and Lemma
\ref{lemma:silly}, we see that it suffices to prove Theorem
\ref{thm:main} for those manifolds $N$ which, besides satisfying the
assumptions in the theorem, have empty parabolic locus and degenerate
end.  Assume that we have such a manifold $N$ and let $\lambda \subset
\D\overline N$ be the ending lamination of its unique end.  Denote
also by $\lambda$ some measured lamination supported by $\lambda$ and
choose once and forever a sequence $\{\gamma_i\}$ of simple closed
curves in $\D\overline N$ converging to $\lambda$ in
$\CP\CM\CL(\D\overline N)$.  By Theorem \ref{Canary-masur}, for all
sufficiently large $i$, the pair $(\overline N, \CN(\gamma_i))$ is
acylindrical.  Here, $\CN(\gamma_i)$ is a regular neighborhood of
$\gamma_i$ in $\D\overline N$.  In section \ref{sec:knots} we will
show:

\begin{named}{Proposition \ref{prop:reduction}}
Let $\overline N$ be a compact irreducible and atoroidal submanifold
of $\BS^3$ with connected boundary of genus at least two, and let
$\eta \subset \D\overline N$ be a simple closed curve with
$(\D\overline N, \CN(\eta))$ acylindrical and $\D\overline N \setminus
\eta$ connected. Then there is a sequence of hyperbolic knot
complements $\{M_{K_i}\}$ converging geometrically to a hyperbolic
manifold $N_\eta$ homeomorphic to the interior of $\overline N$, such
that $\eta$ represents a parabolic element in $N_\eta$.
\end{named}

With $\gamma_i$ as above and $i$ sufficiently large, let
$N_{\gamma_i}$ be the hyperbolic 3--manifold provided by Proposition
\ref{prop:reduction}.  The content of Proposition
\ref{prop:limit-degenerate}, proved in section \ref{sec:kleinian}, is
that the manifolds $N_{\gamma_i}$ converge geometrically to $N$.

\begin{named}{Proposition \ref{prop:limit-degenerate}}
Suppose $M$ is a hyperbolic 3--manifold with empty parabolic locus and
a single end.  Assume that the end of $M$ is degenerate with ending
lamination $\lambda$ and suppose $\{\gamma_n\}$ is a sequence of
simple closed curves in $\D\overline{M}$ converging in
$\mathcal{PML}(\D\overline{M})$ to a projective measured lamination
supported by $\lambda$.  If $\{M_n\}$ is any sequence of hyperbolic
3--manifolds homeomorphic to $M$, such that $\gamma_n$ is parabolic in
$M_n$ for all $n$, then $M$ is a geometric limit of the sequence
$\{M_n\}$.
\end{named}

By Proposition \ref{prop:reduction}, each of the $N_{\gamma_i}$ is a
limit of hyperbolic knot complements and by Proposition
\ref{prop:limit-degenerate} the $N_{\gamma_i}$ converge geometrically
to $N$.  Lemma \ref{lemma:silly} concludes the proof of Theorem
\ref{thm:main}.
\end{proof}

%%%%%%%%%%%%%%%%%%%%%%%%%%%%%%%%%%%%%%%%%%%%%%%%%%%%%%%%%%%%%%%%%

\section{Three facts on geometric limits}\label{sec:kleinian}

The main goals of this section are Proposition
\ref{prop:only-degenerate} and Proposition
\ref{prop:limit-degenerate}.  We also prove a technical result needed
in the proof of Proposition \ref{prop:reduction} later on.

\subsection{Approximating by manifolds with degenerate ends}\label{sec:degenerate}
The statement of Proposition \ref{prop:only-degenerate} below is
summarized by saying that every hyperbolic manifold satisfying the
assumptions of Theorem \ref{thm:main} can be approximated by a
sequence of manifolds which, besides satisfying the same assumptions,
have the property that their unique end is degenerate.

\begin{prop}{\label{prop:only-degenerate}}
Let $N$ be a complete, infinite volume, hyperbolic 3--manifold with
non--abelian, finitely generated fundamental group and a single
topological end.  Assume that $N$ is homeomorphic to a submanifold of
$\BS^3$.  Then $N$ is a geometric limit of a sequence $\{N_i\}$ of
hyperbolic 3--manifolds such that for all $i$ the following holds:
\begin{itemize}
\item $N_i$ has finitely generated fundamental group, no parabolics
and a single end,
\item the end of $N_i$ is degenerate, and
\item $N_i$ admits an embedding into $\BS^3$.
\end{itemize}
\end{prop}

\begin{proof}
Observe that it follows from the Strong Density Theorem that the
manifold $N$ can be approximated by homeomorphic convex cocompact
manifolds.  In other words, we may assume that $N$ is convex
cocompact.

Below we will use an argument of Brooks \cite{Brooks} to prove:

\begin{lem}\label{prop:cocompact-by-maxcusps}
Let $N$ be a complete, convex cocompact hyperbolic 3--manifold with
finitely generated fundamental group and a single topological end.
Assume that $N$ is homeomorphic to a submanifold of $\BS^3$.  Then $N$
is a geometric limit of the sequence of convex cores $\{CC(N_i)\}$
associated to a sequence of hyperbolic 3--manifolds $N_i$ such that
the following holds:
\begin{itemize}
	\item $N_i$ has finitely generated fundamental group and a single
	end,
	\item $N_i$ admits an embedding into $\BS^3$, and 
	\item $N_i$ is a maximal cusp.
\end{itemize}
\end{lem}

Assuming Lemma \ref{prop:cocompact-by-maxcusps} we conclude the proof
of Proposition \ref{prop:only-degenerate}.  Let $N_i$ be one of the
maximal cusps provided by Lemma \ref{prop:cocompact-by-maxcusps}.  By
work of Canary, Culler, Hersonsky, and Shalen \cite[Lemma 15.2]{cchs},
$N_i$ is the algebraic limit of a sequence $\{ N_n^i\}_n$ of
hyperbolic 3--manifolds homeomorphic to $M$, without parabolics, and
whose only end is degenerate.  By Proposition \ref{prop:accs}, there
is a subsequence of $\{N_n^i\}_n$, say the whole sequence, converging
geometrically to some manifold $N_\infty^i$ into which the convex core
$CC(N_i)$ of the maximal cusp $N_i$ can be embedded. Since the convex
cores $CC(N_i)$ converge geometrically to $N$, and each one of them is
contained in a geometric limit of the sequence $\{N_n^i\}$, we deduce
that there is a diagonal sequence $\{N^i_{n_i}\}$ converging
geometrically to $N$.  This diagonal sequence satisfies the
requirements of the proposition.
\end{proof}

It remains to prove Lemma \ref{prop:cocompact-by-maxcusps}.  Before
launching the proof we need to remind the reader of a couple of
definitions.  Assume for the sake of concreteness that
$N=\BH^3/\Gamma$ is convex cocompact.  Recall that the {\em
discontinuity domain} $\Omega_\Gamma = \D_\infty \BH^3 \setminus
\Lambda_\Gamma$ is the complement of the limit set $\Lambda_\Gamma$ of
$\Gamma$ in the boundary at infinity $\D_\infty\BH^3$.  Identifying
$\D_\infty \BH^3 = \BC P^1$ we obtain that the surface $\D_\infty N =
\Omega_\Gamma/\Gamma$ is not only a Riemann surface but has also a
canonical projective structure.  A \emph{circle in $\D_\infty N$} is a
topological circle which is actually round with respect to the
canonical projective structure.  A \emph{circle packing} is a
collection of circles in $\D_\infty N$ bounding disjoint disks, such
that every component of the complement of the union of all these disks
are curvilinear triangles.  We will derive Lemma
\ref{prop:cocompact-by-maxcusps} from the following fact due to Brooks
\cite{Brooks}:
 
\begin{sat}[Brooks]\label{prop:circle-packing}
For all $\epsilon>0$ there is a convex cocompact hyperbolic
3--manifold $N_\epsilon$ homeomorphic to $N$ such that the following
holds: $N_\epsilon$ is $(1+\epsilon)$--bilipschitz to $N$, and
$\D_\infty N_\epsilon$ admits a circle packing such that each disk has
at most diameter $\epsilon$ with respect to the canonical hyperbolic
metric of $\D_\infty N$.
\end{sat}

This theorem is not stated exactly in this way in \cite{Brooks} but it
follows easily from the arguments used to prove Theorem 3 therein.

\begin{proof}[Proof of Lemma \ref{prop:cocompact-by-maxcusps}]
Let $N$ be convex cocompact, and let $\D_\infty N$ be its conformal
boundary with induced projective structure.  Choose $\epsilon_i$
positive and tending to $0$, and for all $i$ let $N_i^1$ be the
manifold provided by Theorem \ref{prop:circle-packing}, and let
$\CC_i$ be the corresponding circle packing of $\D_\infty N_i^1$.
Observe that $N$ is a geometric limit of the sequence $N_i^1$.

Each of the circles in the circle packing $\CC_i$ bounds a properly
embedded totally geodesic plane in $N_i^1$.  Paint such a plane black.
Any two of these black planes either coincide or are disjoint.
Moreover observe that if $x$, $y$, and $z$ are the vertices of one of
the (triangular) interstices of the packing $\CC_i$, say bounded by
the circles $C_1$, $C_2$, and $C_3$, then the ideal triangle with
vertices $x$, $y$, $z$ is perpendicular to the corresponding black
planes $P_1$, $P_2$, $P_3$.  Paint carefully the triangles red.

Since each of the circles has at most diameter $\epsilon_i \to 0$, for
all $d>0$ there is some $i_d$ such that for $i \ge i_d$, none of the
disks or triangles enters the radius $d$ neighborhood of the convex
core $CC(N_i^1)$ of $N_1^1$.  Cut $N_i$ open along all these planes
and triangles and let $N_i^2$ be the closure of the component
containing $CC(N_i^1)$.

The boundary of $N_i^2$ consists of totally geodesic pieces, some
black and some red.  Let $N_i^3$ be the double of $N_i^2$ along the
black boundary, which consists of subsets of the disks.  Since the red
ideal triangles are perpendicular to the black planes of the boundary
of $N_i^2$, the manifold $N_i^3$ has totally geodesic (red) boundary.
Moreover, since $N^3_i$ is obtained by doubling $N_i^2$, each of the
boundary components of $N^3_i$ is the double of one of the red
triangles in the boundary of $N_i^2$.  Thus every boundary component
of $N_i^3$ is a 3--punctured sphere.  Hence $N_i^3$ is the convex core
of a maximal cusp for all $i$.

Additionally, observe that the manifold compactification of $N_i^3$ is
homeomorphic to two copies of $\overline N$ to which we have attached
1--handles, one for each circle in the circle packing.  In particular,
$N_i^3$ embeds in $\BS^3$ and has a single end.

By construction, $N$ is a geometric limit of the sequence $N_i^1$.
The condition that for all $d$, black disks and red triangles are
eventually of distance greater than $d$ from $CC(N_i^1)$ implies that
$N$ is also a geometric limit of the sequence $N_i^2$.  Each $N_i^2$
embeds into $N_i^3$.  Thus $N$ is a geometric limit of the manifolds
$N_i^3$ as well.
\end{proof}

\begin{bem}
Observe that in the proof of Lemma \ref{prop:cocompact-by-maxcusps} we
used the fact that the manifold $N$ had only a single end to conclude
that the manifolds $N^3_i$ embed into $\BS^3$.  For manifolds with
more ends this argument fails.  However, in some cases it should be
possible to by-pass this problem by choosing the circle packings
$\CC_i$ with more care than we did.  For instance, if $N$ admits a
fixed--point free involution $\tau$ which preserves each boundary
component, then we can consider only $\tau$-equivariant circle
packings and glue each black tile with its image under $\tau$.  This
shows for instance that whenever $\Gamma\subset\PSL_2\BR$ is a torsion
free, cocompact Fuchsian group such that the quotient surface
$\BH^2/\Gamma$ admits an orientation preserving fixed--point free
involution, then the hyperbolic 3--manifold $\BH^3/\Gamma$ is the
geometric limit of a sequence of hyperbolic 3--manifolds $N_i$,
homeomorphic to submanifolds of $\BS^3$, with two topological ends
which are degenerate.
\end{bem}

%%%%%%%%%%%%%%%%%%%%%%%%%%%%%%%%%%%%%%%%%%%%%%%%%%%%%%%%%%%%%%%%%
\subsection{Approximating by manifolds with prescribed cusps}\label{sec:construct}
We show now that given a manifold $M$ with a single degenerate end, we
may obtain $M$ as a limit of hyperbolic manifolds with few
restrictions on their geometry.

\begin{prop}{\label{prop:limit-degenerate}}
Suppose $M$ is a hyperbolic 3--manifold with empty parabolic locus and
a single end. Assume that the end of $M$ is degenerate with ending
lamination $\lambda$ and suppose $\{\gamma_n\}$ is a sequence of
simple closed curves in $\D\overline{M}$ converging in
$\mathcal{PML}(\D\overline{M})$ to a projective measured lamination
supported by $\lambda$.  If $\{M_n\}$ is any sequence of hyperbolic
3--manifolds homeomorphic to $M$, such that $\gamma_n$ is parabolic in
$M_n$ for all $n$, then $M$ is a geometric limit of the sequence
$\{M_n\}$.
\end{prop}

\begin{proof}
Observe that it suffices to show that every subsequence of the
sequence $\{M_n\}$ contains a further subsequence which converges
geometrically to $M$.  In particular, we may pass to subsequences as
often as we wish.

Abusing notation, let $\lambda$ also denote the measured lamination
with support $\lambda$ which is the limit of the $\gamma_n$.  By
Theorem \ref{Canary-masur}, the measured lamination $\lambda$ is
doubly incompressible.  In particular, it follows from the work of
Kim, Lecuire, and Ohshika \cite[Theorem 2]{KLO} that every subsequence
of $\{M_n\}$ has a subsequence, say the whole sequence, which
converges algebraically to a hyperbolic 3--manifold $M_A$ homotopy
equivalent to $\overline{M}$.

If $\lambda$ were realized in $M_A$, then by Proposition
\ref{realized-continuity}
$$\lim_nl_{M_n}(\gamma_n)=\infty.$$ On the other hand, since
$\gamma_n$ is parabolic in $M_n$ we have that
$$l_{M_n}(\gamma_n)=0$$ for all $n$.  Hence $\lambda$ cannot be
realized.  It now follows from Proposition \ref{detect-ending} that
$M_A$ is homeomorphic to $M$, has no parabolics, and its end is
degenerate with ending lamination $\lambda$.  In particular,
Proposition \ref{alg-strong} implies that the sequence $M_n$ converges
strongly to $M_A$.  Moreover, since $M$ and $M_A$ have the same ending
lamination $\lambda$, the ending lamination theorem implies that $M$
and $M_A$ are isometric.
\end{proof}

\subsection{Geometric limits of gluings}
We now study some of the geometric limits of sequences of hyperbolic
manifolds obtained by gluing two manifolds by higher and higher powers
of a pseudo-Anosov; compare with \cite{Soma}.  We prove:

\begin{prop}\label{prop:gluings}
Assume that $(\overline N_1, P_1)$ and $(\overline N_2, P_2)$ are
acylindrical, that $S = \D\overline N_1 \setminus P_1$ is connected,
that $\D\overline N_2 \setminus P_2$ has a component $S'$ homeomorphic
to $S$, and fix a homeomorphism $\phi \co S \to S'$.  Fix also a
pseudo-Anosov mapping class $\psi \co S\to S$ and consider %% for all $n$
the 3--manifold
$$N^n=\overline N_1\cup_{\phi\circ\psi^n}\overline N_2$$ obtained by
gluing $\overline N_1$ and $\overline N_2$.  Assume finally that for
each $n$ there is a hyperbolic 3--manifold $M_n$ and a
$\pi_1$--injective embedding $N^n \hookrightarrow M_n$, with all
curves in $P_1 \cup P_2$ mapped to parabolics in $M_n$.

Then some geometric limit $M_G$ of the sequence $\{M_n\}$ is
homeomorphic to the interior of $\overline{N}_1$, has parabolic locus
$P_1$, and its only geometric end is degenerate.
\end{prop}

The proof of Proposition \ref{prop:gluings} is similar to the proof of
Proposition \ref{prop:limit-degenerate}.  This result will play a key
role in the proof of Proposition \ref{prop:reduction}.

\begin{proof}
We are going to prove that every subsequence of the sequence $\{M_n\}$
has some further subsequence converging geometrically to a hyperbolic
3--manifold $M_G$ homeomorphic to the interior of $\overline N_1$,
with parabolic locus $P_1$, and whose unique geometric end is
degenerate with ending lamination equal to the repelling lamination of
$\psi$.  The ending lamination theorem implies then that any two of
these geometric limits are isometric; this shows that $M_G$ is a
geometric limit of the whole sequence $\{M_n\}$.  In particular, as in
the proof of Proposition \ref{prop:limit-degenerate}, we may pass to
subsequences as often as we wish.

Before going any further, observe that the assumption that
$\D\overline N_1 \setminus P_1$ and $\D\overline N_2 \setminus P_2$
are incompressible implies that $\pi_1(\overline N_1)$ injects into
$\pi_1(N^n)$ and hence into $\pi_1(M_n)$.  The assumption that
$(\overline N_1, P_1)$ and $(\overline N_2, P_2)$ are actually
acylindrical implies furthermore:

\begin{lem}\label{lem:blahblah}
Let $\overline N_1$ and $M_n$ be as in the statement of Proposition
\ref{prop:gluings} and identify $\pi_1(\overline N_1)$ with a subgroup
of $\pi_1(M_n)$.  If two elements $\gamma, \gamma' \in \pi_1(\overline
N_1)$ are conjugate in $\pi_1(M_n)$ for some $n$, then they are
conjugate in $\pi_1(\overline N_1)$.  Equivalently, if two essential
curves $\gamma, \gamma' \subset\overline N_1$ are freely homotopic
within $M_n$ they are freely homotopic within $\overline N_1$.  \qed
\end{lem}

Consider now the coverings $M_n^1$ and $M_n^2$ of $M_n$ corresponding
to the subgroups $\pi_1(\overline N_1)$ and $\pi_1(\overline N_2)$ of
$\pi_1(M_n)$ 
respectively.  Since $(\overline N_1,P_1)$ and $(\overline N_2,P_2)$
are acylindrical, Thurston's compactness theorem (see \cite{Kap})
applies.  Hence, passing to a subsequence we may assume that the two
sequences $\{M_n^1\}$ and $\{M_n^2\}$ converge algebraically.  Let
$M_G$ be the algebraic limit of the sequence $\{M_n^1\}$.  Every curve
in $P_1$ represents a parabolic element in $M_G$ because this is the
case in $M_n^1$ for all $n$.  On the other hand, fix a simple closed
curve $\alpha \subset S' \subset \D\overline N_2 \setminus P_2$.  The
compactness of the sequence $\{M_2^n\}$ implies that there is some $L$
with $l_{M_2^n}(\alpha) \le L$ for all $n$.  After the identification
$\phi \circ \psi^n$, the curve $\alpha$ becomes the curve
$\psi^{-n}(\phi^{-1}(\alpha)) \subset S \subset \D\overline N_1
\setminus P_1$.  In particular, we have
$$l_{M_1^n}(\psi^{-n}(\phi^{-1}(\alpha))) \le L\ \ \ \hbox{for all
n.}$$ Let $\lambda$ be the repelling lamination of $\psi$ and observe
that
$$\lim_{n\to\infty} \psi^{-n} (\phi^{-1}(\alpha)) = \lambda\ \ \
\hbox{in}\ \mathcal{PML}(S).$$ A small variation of the argument used
in the proof of Proposition \ref{prop:limit-degenerate} shows that
$\lambda$ is not realized in $M_G$.  Again as in the proof of
Proposition \ref{prop:limit-degenerate}, we derive from Proposition
\ref{detect-ending} that $P_1$ is the parabolic locus of $M_G$ and
that its unique geometric end is degenerate with ending lamination
$\lambda$.  Also as in the proof of Proposition
\ref{prop:limit-degenerate}, Proposition \ref{alg-strong} shows that
the sequence $\{M_n\}$ does not only converge algebraically but also
geometrically to $M_G$.

We have proved that the sequence of covers $M_n^1$ of $M_n$ has the
desired geometric limit.  We claim that this is also the case for the
original sequence $\{M_n\}$.  Passing perhaps to a further
subsequence, we may assume that the manifolds $\{M_n\}$ converge
geometrically to some manifold $M_G'$ covered by $M_G$.  Since the
unique geometric end of $M_G$ is degenerate, we deduce from the
Thurston--Canary covering theorem that the cover $M_G\to M_G'$ is
finite--to--one. We claim that this cover is trivial. Otherwise, there
are two elements $\gamma,\gamma'\in\pi_1(M_G)=\pi_1(\overline N_1)$
which are conjugated in $\pi_1(M_G')$ but not in $\pi_1(M_G)$. In
particular, we deduce from the geometric convergence of the sequence
$\{M_n\}$ to $M_G$ that for all sufficiently large $n$ the elements
$\gamma,\gamma' \in \pi_1(\overline N_1) \subset \pi_1(M_n)$ are also
conjugated in $\pi_1(M_n)$. This contradicts Lemma \ref{lem:blahblah}.
Thus the covering $M_G\to M_G'$ is trivial and hence $M_G' = M_G$. In
other words, $M_G$ is a geometric limit of the sequence $\{M_n\}$.
\end{proof}

%%%%%%%%%%%%%%%%%%%%%%%%%%%%%%%%%%%%%%%%%%%%%%%%%%%%%%%%%%%%%%%%%
\section{Constructing some limits of knot complements}\label{sec:knots}

The goal of this section is to prove Proposition \ref{prop:reduction}.
In order to do so, we start studying certain multicurves in the
boundary of a compression body.  These multicurves are used to prove
first a version of Proposition \ref{prop:reduction} for links.  The
desired knots are then obtained from these links by Dehn-filling.

\subsection{Curves on a compression body}
We prove now a topological fact used in the proof of Proposition
\ref{prop:links}.  The setting is the following: Let $\Sigma$ be a
closed surface and $C$ the compression body obtained by gluing $\Sigma
\times [0,1]$ and $\BS^1 \times \BD^2$ along disks $D_1 \subset \Sigma
\times \{1\}$ and $D_2 \subset\D(\BS^1\times \BD^2)$.  We denote by
$\D_eC$ the compressible boundary component of $C$, by $\D_iC$ the
incompressible one, and by $D$ the disk $D_1=D_2\subset C$.

Fix a simple curve $\alpha \subset \D_eC$ disjoint from $D$, contained
in the component of $C \setminus D$ corresponding to $\BS^1 \times
\BD^2$, representing a generator of $\BZ = \pi_1(\BS^1 \times
\BD^2)$. We will use repeatedly the following fact, which is easily
verified by considering the two components of $C\setminus D$.

\begin{lem}\label{lem:ess-disks}
	The only embedded essential disks in $C\setminus D$ are isotopic to
	$D$ or intersect $\alpha$. \qed
\end{lem}

Let $\gamma$ be any simple closed curve in $\D_eC$ which intersects
$\alpha$ exactly once, with $i(\gamma,\D D ) >0$, and such that, after
isotoping $\gamma$ so that it intersects $\D D$ minimally, $\gamma
\cap (\Sigma \times \{1\} \setminus D_1)$ contains at least two
nonisotopic properly embedded arcs.  Here $i(\cdot,\cdot)$ is the
geometric intersection number.  Finally, denote by $\beta$ the
boundary of a regular neighborhood of $\alpha \cup \gamma$. We prove:

\begin{prop}\label{prop:curves}
Let $\alpha$ and $\beta$ be as above.  Then there exists a
non-separating curve $\eta \subset \D_iC$ such that $(C, \CN(\alpha
\cup \beta \cup \eta))$ is acylindrical.
\end{prop}

\begin{proof}
We first prove prove that the multicurve $\alpha \cup \beta \subset
\D_eC$ intersects every properly embedded essential disk, M\"obius
band, and annulus $(A,\D A) \subset (C,\D_eC)$.
\medskip

\noindent{\bf Claim 1.} The multicurve $\alpha\cup\beta$ intersects
every properly embedded essential disk in $C$.

\begin{proof}[Proof of Claim 1]
%% \emph{(Actually, what is proved here is that every essential disk in
%% 	$C$ is either isotopic to $D$ or intersects $\alpha$.)}

Let $\Delta$ be a properly embedded essential disk in $C$ such that
$\D \Delta$ does not meet $\alpha \cup \beta$.

If $\Delta \cap D = \emptyset$, then $\Delta$ is an embedded essential
disk in $C\setminus D$, hence by Lemma \ref{lem:ess-disks}, $\Delta$
is isotopic to $D$ or meets $\alpha$.  By assumption it does not meet
$\alpha$.  But $i(\beta, \D D) = 2i(\gamma, \D D) > 0$, by choice of
$\beta$ (and $\gamma$), hence $\Delta$ cannot be isotopic to $D$.
Thus $\Delta \cap D \neq \emptyset$.

Consider the intersections $\Delta \cap D$.  These consist of closed
curves and arcs.  Using the irreducibility of $C$, along with an
innermost disk argument, we may assume no components of $\Delta \cap
D$ are closed curves (else they can be isotoped off).  So consider
arcs of intersection.  There is an innermost arc of intersection on
$\Delta$, which bounds a disk $E_1 \subset \Delta$ disjoint from $D$,
and some disk $E_2 \subset D$.  Together, these two disks form a disk
$E$ in $C$.  By pushing $E_2$ off $D$ slightly, we may assume $E$ is
embedded in $C\setminus D$.

If $E$ is not essential, then together $E_1$ and $E_2$ and their
boundary curve bound a ball in $C$.  We may isotope $E_1$ through this
ball to decrease the number of intersections of $\Delta$ and $D$.

So assume $E$ is essential.  Then by Lemma \ref{lem:ess-disks}, $E$ is
isotopic to $D$ or meets $\alpha$.  First, $E$ cannot meet $\alpha$,
for by assumption $E_1$ doesn't meet $\alpha$, and $E_2$ is a subset
of $D$, which does not meet $\alpha$.  So $E$ is isotopic to $D$.  But
then again we may isotope $E_1$ through $D$, reducing the number of
intersections of $\Delta$ and $D$.

Repeating this argument a finite number of times, we find $\Delta \cap
D$ must be empty, which is a contradiction.
\end{proof}

{\bf Claim 2.} The multicurve $\alpha\cup\beta$ intersects every
properly embedded M\"obius band in $(C,\D_eC)$ .

\begin{proof}[Proof of Claim 2]
Let $(M,\D M)\subset(C,\D_eC)$ be a properly embedded M\"obius
band. Since $C$ is orientable, $M$ has to be one-sided. In particular,
the homology class $[M]\in H_2(C,\D_eC;\BZ/2\BZ)$ is non-trivial; by
duality, there is some class $[c]$ in $H_1(C;\BZ/2\BZ)$ with
$[c]\cap[M]=1\mod 2$. The first homology of $C$ is generated by
$H_1(\D_iC)$ and $[\alpha]$. Since $\D_iC\cap M=\emptyset$, we deduce
that $\alpha$ has to intersect $M$.
\end{proof}

{\bf Claim 3.}  The multicurve $\alpha\cup\beta$ intersects every
properly embedded essential annulus $(A, \D A) \subset (C, \D_eC)$
which is disjoint of $D$.

\begin{proof}[Proof of Claim 3]
Let $(A, \D A)$ be an essential annulus in $(C, \D_eC)$, and assume
$\D A \cap (\alpha \cup \beta) = \emptyset$ and $A\cap D=\emptyset$.

First, if $A$ is contained in the component corresponding to
$\BS^1\times\BD^2$, then because $\D A \cap \alpha = \emptyset$, $\D
A$ must be parallel to $\alpha$ or $\D D$.  In either case, $A$ will
be inessential.  Thus $A$ will lie in the component of $C\setminus D$
corresponding to $\Sigma \times [0,1]$. In particular, the annulus $A$
is isotopic, relative to its boundary, to an annulus $A'$ contained in
$\Sigma\times\{1\}$. The assumption that $A$ is essential in $C$
implies that $D_1\subset A'$ (recall $D_1 \subset \Sigma \times \{1\}$
is the disk $D_1 = D$ in $C$).

Cut open the surface $\D_eC$ along $\D A' = \D A$ and let $X$ be the
component containing $\D D$. Observe that $Y = X \cap (\Sigma \times
\{1\} \setminus D_1)$ is a pair of pants.  By assumption, the curve
$\beta$ is disjoint of $\D A$ and by construction intersects $\D D$.
In particular, $\beta\subset X$.  This implies that the curve $\gamma$
is also contained in $X$.  By assumption, $\gamma \cap (\Sigma \times
\{1\} \setminus D_1)$ contains at least two nonisotopic properly
embedded essential arcs.  Hence $\gamma \cap Y$ contains at least two
nonisotopic arcs.  This is impossible, since $Y$ is a pair of pants
and all the arcs in $\gamma \cap Y$ end in the same boundary
component.  This contradiction concludes the proof of Claim 3.
\end{proof}

{\bf Claim 4.}  The multicurve $\alpha \cup \beta$ intersects every
properly embedded essential annulus $(A, \D A) \subset (C, \D_eC)$.

\begin{proof}[Proof of Claim 4]
Let $(A, \D A)$ be an essential annulus in $(C, \D_eC)$, and assume
$\D A \cap (\alpha \cup \beta) = \emptyset$.  By Claim 3, we may
assume that $A \cap D \neq \emptyset$.  As in the proof of Claim 1,
consider the arcs and curves of intersection.

Because $A$ is essential, and any curve of intersection bounds a disk
in $D$, it must bound a disk in $A$, and thus we may isotope $A$ so
that there are no closed curves of intersection with $A$.

Suppose there is an arc $\tau$ of $A \cap D$ such that both endpoints
of $\tau$ lie on the same component of $\partial A$.  Then $\tau$,
along with a portion of $\D A$, bounds a disk in $A$.  The only
components of $A\cap D$ that lie in that disk will also have endpoints
on the same component of $\D A$.  Thus we may assume $\tau$ is an
innermost arc of intersection, which, together with a portion of $\D
A$, bounds a disk in $A$ disjoint from $D$.  The arc $\tau$ also
bounds a disk on $D$.  The union of these two disks can be pushed off
of $D$ slightly to give an embedded disk $E$ in $C\setminus D$.

Now, if $E$ is inessential, we can isotope $A$ through $D$ to reduce
the number of intersections.  So assume $E$ is essential.  Then by
Lemma \ref{lem:ess-disks}, $E$ is isotopic to $D$ or meets $\alpha$.
$E$ cannot meet $\alpha$, for neither $A$ nor $D$ meet $\alpha$.  Thus
$E$ is isotopic to $D$.  But then again we may isotope $A$ through $D$
and reduce the number of intersections $A\cap D$.

So we may assume no arcs of $A\cap D$ have both endpoints on the same
component of $\D A$.

%% Next, suppose there is a closed curve $\sigma$
%% of intersection of $A\cap D$.  If $\sigma$ bounds a disk in $A$, then
%% again an innermost disk argument implies we may push off this
%% intersection.  So assume $\sigma$ is isotopic to the core curve of
%% $A$.  Then there is an outermost such curve $\sigma'$ which, together
%% with a component of $\D A$, bounds a subannulus $A_1 \subset A$ such
%% that $A_1 \cap D =\sigma'$ (using here the fact that no arcs of
%% intersection of $A\cap D$ have both endpoints on the same component of
%% $\D A$).  But then $\sigma'$ bounds a disk $D'\subset D$.  The union
%% $E = A_1 \cup D'$ is a properly embedded disk, disjoint of $\alpha
%% \cup \beta$. By Claim 1, $E$ bounds a ball in $C$ and hence one of the
%% components of $\D A$ bounds a disk in $\D_e C$.  This contradicts the
%% assumption that $A$ is essential.  So we may assume there are no
%% curves of intersection of $A\cap D$.

Thus assume each arc of intersection has endpoints on distinct
components of $\D A$.  Let $\tau_1$ and $\tau_2$ be consecutive arcs
of intersection, such that they bound a disk on $A$ disjoint from $D$.
Note $\tau_1$ and $\tau_2$ will also bound disjoint disks on $D$.
Putting these disks together and pushing off $D$ slightly, again we
have a disk $E$ embedded in $C\setminus D$.

If $E$ is essential in $C\setminus D$, then again Lemma
\ref{lem:ess-disks} gives a contradiction:  $E$ cannot meet $\alpha$
because $A$ and $D$ do not, and if $E$ is isotopic to $D$ then we can
isotope $A$ through $D$, reducing the number of intersections.

So $E$ is inessential in $C\setminus D$.  But then we may isotope the
portion of $A$ between $\tau_1$ and $\tau_2$ to lie on the boundary
$\D_eC$.  This is true of any consecutive arcs $\tau_1$ and $\tau_2$
on $A$.  Since $A$ can be cut into pieces, each of which is bounded by
consecutive arcs, $A$ is not an essential annulus.  This contradiction
finishes the proof of the claim.
\end{proof}

We can now conclude the proof of Proposition \ref{prop:curves}. Let
$\CN(\alpha \cup \beta)$ be a regular neighborhood of the multicurve
$\alpha \cup \beta$.  By Claim 1, $\D C \setminus \CN(\alpha \cup
\beta)$ is incompressible.  Assume that $(C, \CN(\alpha \cup \beta))$
is not acylindrical.  Then we can consider the JSJ--splitting of $C$
relative to a regular neighborhood of $\CN(\alpha \cup \beta)$; let
$S$ be the union of the Seifert pieces.  Since $C$ is not itself
Seifert fibered, $S$ is a proper subset of $C$.  The boundary of $S$
consists of a collection of properly embedded essential annuli; it is
easy to see that this implies that the interior boundary $\D_iC =
\Sigma \times \{0\}$ of $C$ is not contained in $S$.  On the other
hand, all the boundary annuli of $S$ have at least one of their
boundary components in $\D_iC$.  Let $X$ be the boundary in $\D_iC$ of
$\D_iC \cap S$.  Any curve $\eta$ which, together with $X$, fills the
surface $\D_iC$ has the property of intersecting every properly
embedded essential annulus in $(C, \CN(\alpha \cup \beta))$. In
particular, $(C, \CN(\alpha \cup \beta \cup \eta))$ is acylindrical;
this concludes the proof of Proposition \ref{prop:curves}.
\end{proof}

\subsection{Limits of link complements}

In this section we prove a version of Proposition \ref{prop:reduction}
for links instead of knots.  As mentioned above, we will obtain the
desired knots needed to prove Proposition \ref{prop:reduction} using
Dehn filling.  In order to be able to do so, we need to construct
links satisfying certain conditions.  One of those conditions is that
a certain slope has a long \emph{normalized length}, as in \cite{HK}.

\begin{defi*}
Let $M$ be a hyperbolic 3--manifold with a rank 2 cusp $T$, and let
$H$ be an embedded horoball neighborhood of the cusp with boundary $\D
H$.  Let $s$ be a \emph{slope} on the cusp , that is, an isotopy class
of simple closed curves on $T$.  The \emph{normalized length} of $s$
is defined to be the length of a geodesic representative of $s$ on $\D
H$, divided by the square root of the area of the torus $\D H$.  Note
that this definition is independent of choice of horoball neighborhood
$H$.
\end{defi*}

We prove now:

\begin{prop}\label{prop:links}
Let $\overline N$ be a compact irreducible and atoroidal submanifold
of $\BS^3$ with connected boundary of genus at least 2, and let $\eta
\subset \D\overline N$ be a simple closed curve with $(\D\overline N,
\CN(\eta))$ acylindrical and $\D\overline N \setminus \eta$
connected.  Then there is a sequence of hyperbolic link complements
$M_{L_i}$ converging geometrically to a hyperbolic manifold $N_\eta$
homeomorphic to the interior of $\overline N$ and such that $\eta$
represents a parabolic element in $N_\eta$.

Moreover, the links $L_i$ have exactly four components, $\alpha_i$,
$\beta_i$, $\kappa_i$, and $\eta_i$.  The normalized length of the
standard meridian on $\eta_i$ is increasing in an unbounded manner,
and $\alpha_i$ and $\beta_i$ bound disjoint disks in the complement of
$\eta_i$, $\alpha_i$ and $\beta_i$ in $\BS^3$.
\end{prop}

The construction given in the proof of Proposition \ref{prop:links} is
quite involved.  There are simpler ways to build $N_\eta$ as the
geometric limit of link complements.  However, our choice of link
complements and of $N_\eta$ needs the extra care so that we may turn
these links into knots to prove Proposition \ref{prop:reduction} in
the following section.
\medskip

Before launching the proof of Proposition \ref{prop:links}, let
$\Sigma$ be a closed surface homeomorphic to $\D\overline N$, $C$ the
compression body considered in the previous section and $\alpha, \beta
\subset \D_e C$ and $\eta' \subset \D_iC$ the curves provided by
Proposition \ref{prop:curves}.  Let also $\phi\co \D\overline N \to \D_iC$
be any homeomorphism with $\phi(\eta)= \eta'$.  Denote by $X_\phi$ the
manifold obtained by gluing $\overline N$ and $C$ via $\phi$ and
removing a regular neighborhood of the the curve $\eta = \eta'$.  For
the sake of book--keeping we denote by $\alpha_\phi$ and $\beta_\phi$
the curves in $\D X_\phi$ corresponding to $\alpha$ and $\beta$.
Observe that the pair $(X_\phi, \CN(\alpha_\phi \cup \beta_\phi))$ is
acylindrical.

\begin{lem}\label{lemma:key}
The embedding $\iota\co \overline N \hookrightarrow \BS^3$ extends to an
embedding $\iota'\co X_\phi \hookrightarrow \BS^3$ with the property
that the curves $\iota'(\alpha_\phi)$ and $\iota'(\beta_\phi)$ bound
properly embedded disjoint disks in the closure of $\BS^3 \setminus
\iota'(X_\phi)$.
\end{lem}

\begin{proof}
Using the embedding $C \hookrightarrow X_\phi$, identify the disk $D$
with a disk in $X_\phi$, and let $Y_\phi$ be the component of $X_\phi
\setminus D$ containing $\overline N$.

The manifold $Y_\phi$ is homeomorphic to the complement in $\overline
N$ of the regular neighborhood of some curve contained in the interior
of $\overline N$ and isotopic to $\eta$.  In particular, the embedding
$\iota \co \overline N \hookrightarrow \BS^3$ extends to an embedding
$\hat\iota \co Y_\phi \hookrightarrow \BS^3$.

The other component, say $U$, of $X_\phi \setminus D$ containing
$\alpha$, is homeomorphic to $\BS^1 \times \BD^2$.  In particular, we
can embed $U$ in $\BS^3 \setminus \hat\iota(Y_\phi)$ in such a way
that the image of the curve $\alpha$ bounds a properly embedded disk
$\Delta$ in the complement of $U$ and $\hat\iota(Y_\phi)$.

In order to extend these two embeddings to an embedding $\iota'\co
X_\phi \hookrightarrow \BS^3$ we map the 1--handle joining $Y_\phi$
and $U$ to a 1--handle in $\BS^3 \setminus (\hat\iota(Y_\phi) \cup U)$
whose core is disjoint from $\Delta$.

So far, we have an embedding $\iota'\co X_\phi \hookrightarrow \BS^3$
with the properly that $\iota'(\alpha_\phi)$ bounds an embedded disk
$\Delta$ in the complement of the image of $\iota'$.  Recall that the
curve $\beta$ is obtained as the boundary of the regular neighborhood
of $\alpha \cup \gamma$ where $\gamma$ intersects $\alpha$ once.  The
boundary of a regular neighborhood in $\BS^3 \setminus \iota'(X_\phi)$
of $\iota'(\gamma) \cup \Delta$ is a properly embedded disk with
boundary $\beta$, which is disjoint of $\Delta$.
\end{proof}

\begin{bem}
Neither the constructed embedding $\iota'$, nor the image of the curve
$\alpha_\phi$ depend on $\phi$.  However, the curve
$\iota'(\beta_\phi)$ is very sensitive to $\phi$.
\end{bem}

We are now ready to prove Proposition \ref{prop:links}:

\begin{proof}[Proof of Proposition \ref{prop:links}]
Choose a homeomorphism $\phi$ as above once and for ever.  Choose a
mapping class $\psi \co \D_iC \to \D_iC$ with $\psi(\eta') = \eta'$
and whose restriction to $\D_iC \setminus \eta'$ is pseudo-Anosov.
Consider the sequence of manifolds $X_{\psi^n \circ \phi}$.

For all $n$, choose once and for ever a knot $\kappa_n \subset \BS^3
\setminus \iota'(X_{\psi^n \circ \phi})$ which intersects every
properly embedded essential disk, annulus and M\"obius band therein.
Consider the link
$$L_n = \eta \cup \iota'(\alpha_{\psi^n \circ \phi}) \cup
\iota'(\beta_{\psi^n \circ \phi}) \cup \kappa_n$$ By choice of
$\kappa_n$, the complement $M_{L_n} = \BS^3 \setminus L_n$ of this
link is irreducible and atoroidal.  In particular, by Thurston's
hyperbolization theorem (see \cite{Kap,Otal96,Otal-Haken}), $M_{L_n}$
admits a complete hyperbolic metric.

Observe that by construction, $\overline N$ (minus a small regular
neighborhood of the boundary) is a $\pi_1$-injective submanifold of
$M_{L_n}$.  Moreover, it follows from the construction that the
sequences of manifolds $M_{L_n}$ and submanifolds $\overline N \subset
M_{L_n}$ satisfy the conditions of Proposition \ref{prop:gluings}.

It follows that the sequence $\BS^3 \setminus L_n$ has a geometric
limit $N_\eta$ with the following property:
\begin{itemize}
\item[(*)] $N_\eta$ is homeomorphic to the interior of $\overline N$,
$\eta$ is parabolic in $N_\eta$ and the only geometric end of $N_\eta$
is degenerate.
\end{itemize}

%% REWRITE THE FOLLOWING PARAGRAPH MORE CAREFULLY.

It remains only to show that the normalized length of the standard
meridian of the link component corresponding to $\eta$ tends to
infinity.  This is due to the fact that these manifolds approach
$N_\eta$, whose end is degenerate, and hence in the geometric limit,
the link components tending to $\eta$ tend to a rank 1 cusp, which is
an infinite annulus.  This rank 1 cusp has fixed translation length,
but unbounded length in another direction.

We claim that the unbounded direction corresponds to the limit of
standard meridians of $\eta$ in $M_{L_n}$.  That is, note a standard
meridian of $\eta$ in $M_{L_n}$ intersects $\D\overline N$ twice in
$\iota'(X_{\psi^n \circ \phi})$.  The geometric limit $N_\eta$ is
homeomorphic to the interior of $\overline N$, hence a curve meeting
$D\overline N$ twice must have unbounded length.

Because the translation length of the rank 1 cusp is fixed, lengths of
curves meeting the meridian once have bounded length for large $n$ in
$M_{L_n}$.  Hence the normalized length of the meridian must become
arbitrarily long.
\end{proof}

\subsection{Proof of Proposition \ref{prop:reduction}}
The goal of this section should be clear from the title.

\begin{prop}\label{prop:reduction}
Let $\overline N$ be a compact irreducible and atoroidal submanifold
of $\BS^3$ with connected boundary of genus at least 2, and let $\eta
\subset \D\overline N$ be a simple closed curve with $(\D\overline N,
\CN(\eta))$ acylindrical and $\D\overline N \setminus \eta$ connected.
Then there is a sequence of hyperbolic knot complements $\{M_{K_i}\}$
converging geometrically to a hyperbolic manifold $N_\eta$
homeomorphic to the interior of $\overline N$ and such that $\eta$
represents a parabolic element in $N_\eta$.
\end{prop}

\begin{proof}
By Proposition \ref{prop:links}, we may assume there is a sequence of
hyperbolic link complements $\{M_{L_i}\}$ converging geometrically to
$N_\eta$ homeomorphic to the interior of $\overline N$ and such that
$\eta$ represents a parabolic element in $N_\eta$.  Moreover, we know
that the link $L_i$ has four components $\alpha_i$, $\beta_i$,
$\kappa_i$, and $\eta_i$, that the normalized length of the standard
meridian of $\eta_i$ grows unbounded, and that $\alpha_i$ and
$\beta_i$ bound disjoint disks in the complement of $\alpha_i$,
$\beta_i$, and $\eta_i$ in $\BS^3$.  To prove the proposition, we will
perform hyperbolic Dehn filling on the link components corresponding
to $\alpha_i$, $\beta_i$, and $\eta_i$.  We will use the following
version, whose proof is written carefully in Aaron Magid's thesis
\cite[Theorem 4.3]{Magid}, of Hodgson and Kerckhoff's quantified
Dehn-filling theorem \cite{HK}.  See also Bromberg \cite[Theorem
2.5]{Bromberg-nonlocal}.

\begin{sat}\label{Dehn-filling}
Let $J>1$ and $\epsilon$ positive and smaller than the Margulis
constant. Then there is some $L > 16\pi^2 + 2\pi/\epsilon$ such that
if $M$ is a finite volume hyperbolic 3--manifold, and $s \subset
\D\overline M$ is a slope with normalized length at least $L$, then:
\begin{enumerate}
\item The interior $M_s$ of the Dehn filled manifold $\overline M(s)$
obtained from $\overline M$ by surgery along $s$ is hyperbolic,
\item the geodesic $\gamma_s \subset M_s$ isotopic to the core of the
attached solid torus has at most length $(2\pi) / (L^2 - 16\pi^2) <
\epsilon$, and
\item there is a $J$-bi-Lipschitz embedding $\phi\co M \setminus
\BT_{s,\epsilon} \hookrightarrow M_s$ of the complement in $M$ of the
component in $M^{<\epsilon}$ corresponding to the filled cusp into the
Dehn filled manifold $M_s$.
\end{enumerate}
\end{sat}

Let $M_{\alpha_i\cup\beta_i\cup\kappa_i}$ be the manifold obtained
from $M_{L_i}$ by performing surgery on $M_{L_i}$ along the standard
meridian of $\eta_i$.  Since the normalized length of this meridian
grows unbounded, we deduce from Theorem \ref{Dehn-filling} that for
all sufficiently large $i$ the manifold $M_{\alpha_i \cup \beta_i \cup
\kappa_i}$ is hyperbolic and that $N_\eta$ is also a geometric limit
of the sequence $\{M_{\alpha_i\cup\beta_i\cup\kappa_i}\}$.

Fixing now $i$, consider the manifold $M_{\alpha_i\cup\beta_i}$
obtained by filing the standard meridian of the knot $\kappa_i$.  In
other words, $M_{\alpha_i \cup \beta_i} = \BS^3 \setminus(\alpha_i
\cup \beta_i)$.  Since the two link components $\alpha_i$ and
$\beta_i$ bound disjoint embedded disks, we see that the manifold
$M_{\alpha_i \cup \beta_i}$ is homeomorphic to the interior connected
sum $(\BD^2 \times \BS^1) \# (\BD^2 \times \BS^1)$ of two solid tori.
In particular, there are infinitely many ways in which one can Dehn
fill each of the two ends of $M_{\alpha_i\cup\beta_i}$ so that we
obtain the 3--sphere $\BS^3$.  Each one of these fillings yields a new
embedding of $M_{\alpha_i \cup \beta_i}$ into $\BS^3$; observe that
the knot $\kappa_i \subset M_{\alpha_i\cup\beta_i}$ is mapped under
these new embeddings to a knot which may not be isotopic to the
original $\kappa_i$.

Hence for each $i$ the hyperbolic manifold $M_{\alpha_i \cup \beta_i
\cup \kappa_i}$ admits infinitely many Dehn fillings of the cusps
associated to $\alpha_i$ and $\beta_i$ so that the obtained manifold
is homeomorphic to the complement of a knot in $\BS^3$.  Choosing a
sequence of more and more complicated Dehn fillings in each of the
cusps $\alpha_i$ and $\beta_i$, we obtain a sequence
$\{M_{K_{i,j}}\}_j$ of knot complements.  We deduce from Theorem
\ref{Dehn-filling}, or even from the classical version of Thurston's
Dehn filling theorem \cite{thurston}, that for each fixed $i$ the
following holds:
\begin{itemize}
\item For all sufficiently large $j$, say for all $j$, the knot
complement $M_{K_{i,j}}$ is hyperbolic, and
\item the sequence of knot complements $\{M_{K_{i,j}}\}$ converges
geometrically to $M_{\alpha_i \cup \beta_i \cup \kappa_i}$.
\end{itemize}
Since by construction the sequence $\{M_{\alpha_i \cup \beta_i \cup
\kappa_i}\}$ converges to $N_\eta$, the claim of Proposition
\ref{prop:reduction} follows now from Lemma \ref{lemma:silly}.
\end{proof}

%%%%%%%%%%%%%%%%%%%%%%%%%%%%%%%%%%%%%%%%%%%%%%%%%%%%%%%%%%%%%%%%%

\section{Convex submanifolds and the proof of Theorem \ref{yair}}\label{sec:yair}

In this section we prove Theorem \ref{yair}.  However, before doing so
we have to establish a (perhaps well--known or possibly of independent
interest) property of certain convex subsets in 3--dimensional
hyperbolic space.  Essentially we prove that the completion of the
complement of the convex hull of any subset of $\D_\infty\BH^3$ is a
locally CAT(-1) space.  This fact allows us to adapt an argument due
to Lackenby to prove Theorem \ref{yair}.  We end this section with a
few unrelated observation on the complements of convex submanifolds of
hyperbolic 3--manifolds.

\subsection{Convex hulls}
Let $X\subset\D_\infty\BH^3$ be a closed subset in the boundary at
infinity of $\BH^3$ and assume that $X$ contains at least three points.
Let $U$ be a connected component of the complement $\BH^3 \setminus
CH(X)$ of the convex hull $CH(X)$ of $X$ in $\BH^3$. We consider $U$
with the induced interior metric and let $\overline U$ be its metric
completion.

\begin{bem}
Recall that the completion of a subspace of a metric space may be
different from its closure.  Think of the open interval $(0,1)$ in the
circle $\BR/\BZ$.
\end{bem}

\noindent We prove:

\begin{prop}\label{convex-cat}
Let $X \subset \D_\infty\BH^3$ be a closed set in the boundary at
infinity of $\BH^3$ containing at least three points, and $U$ a
connected component of the complement $\BH^3 \setminus CH(X)$ of the
convex hull $CH(X)$ of $X$ in $\BH^3$.  The metric completion
$\overline U$ of $U$ is a locally CAT(-1) manifold with totally
geodesic boundary $\D\overline U = \overline U \setminus U$,
under the path metric inherited from $\BH^3$.
\end{prop}

Before launching the proof of Proposition \ref{convex-cat} we would
like to remark that this is an almost exclusively $3$--dimensional
phenomenon.  For instance, it was pointed out to us by Larry Guth that
if $X$ is any finite set in $\D_\infty\BH^4$ then the complement of
$CH(X)$ is not aspherical and hence cannot be a locally CAT(-1)
manifold.  It would be interesting to determine for which subsets of
$\D_\infty\BH^n$, $n\ge 4$, Proposition \ref{convex-cat} remains
valid.

\begin{proof}
The assumption that $X$ has at least three points implies that $CH(X)$
is either a totally geodesic surface (if $X$ is contained in a round
circle in $\D_\infty\BH^3$) or a convex set with non-empty interior.

Assume that we are in the first case, or equivalently that $CH(X)
\subset \BH^2$.  If $CH(X) = \BH^2$ then each component of $\BH^3
\setminus CH(X)$ is an open halfspace, its closure is a closed
halfspace, and we have nothing to prove.  If $CH(X)$ is a proper
subset of $\BH^2$ then $U = \BH^3\setminus CH(X)$ is connected and
its metric completion $\overline U$ is homeomorphic to the complement
of a regular neighborhood of $CH(X)$ in $\BH^3$.  In particular,
$\overline U \setminus U$ is the double of $CH(X)$.  In fact, the map
$\overline U \to \BH^3$ induces the ``folding the double'' map
$\overline U \setminus U \to CH(X)$.  The double $D\overline U$ of
$\overline U$ is a hyperbolic cone--manifold with empty boundary and
all cone angles equal to $4\pi$.
This implies local uniqueness of geodesics, 
and hence $D\overline U$ is CAT(-1)
(compare with \cite{Gromov-Thurston}).  
The local uniqueness of
geodesics in $D\overline U$ implies that $\overline U$ is totally
convex and $\D\overline U$ totally geodesic
under the path metric on $\overline U$ inherited from the hyperbolic metric.
Since totally convex subsets of locally CAT(-1) spaces are locally
CAT(-1), the claim follows in this case.

In some sense, the case that $CH(X)$ has nonempty interior
is less confusing.  In this case the metric completion $\overline U$
of a component $U$ of $\BH^3 \setminus CH(X)$ is equal to its closure
in $\BH^3$.  Moreover, a theorem of Thurston (see for example
\cite{epstein-marden}) asserts that $\D\overline U = \overline U
\setminus U = \D CH(X)$ is, with respect to its intrinsic distance, a
complete hyperbolic surface.

Now choose a nested collection of finite subsets $X_i$ of $X$ with
dense union.  In other words,
\begin{equation}\label{eq:finite-approx}
X_1 \subset X_2 \subset X_3 \subset \dots \quad \hbox{and} \ \ 
X=\overline{\cup_{i=1}^\infty X_i}
\end{equation}
We may assume without loss of generality that none of the sets $X_i$
is contained in a round circle, since $X$ is not.
For each $i$, the set $U_i = \BH^3 \setminus CH(X_i)$ is connected
and contains $\BH^3 \setminus CH(X)$.  When $i$ tends to $\infty$,
the closures $\overline U_i$ of $U_i$ converge in the pointed
Hausdorff topology to the closure $\overline{\BH^3\setminus CH(X)}$
of $\BH^3\setminus CH(X)$.  Moreover, for every $p\in\BH^3\setminus
CH(X)$ there is $\epsilon>0$ such that for all $i$
sufficiently large, $B_p(\epsilon,\BH^3)\cap\overline U_i$ and
$B_p(\epsilon,\BH^3)\cap (\BH^3\setminus CH(X))$ are simply
connected. In other words, the sequence of universal covers of
$\overline U_i$ converge in the pointed Hausdorff topology to the
universal cover of $\overline{\BH^3\setminus CH(X)}$.  It follows now
from \cite[II, Theorem 3.9]{Bridson-H} that it suffices to show that
each of the $\overline U_i$ is locally CAT(-1).

In other words, in order to conclude the proof of Proposition
\ref{convex-cat} it remains to prove it for finite sets $X \subset
\D_\infty\BH^3$ which are not contained in a round circle.  Under this
assumption, $CH(X)$ is a convex ideal polyhedron with non-empty
interior.  In particular $CH(X)$ has finitely many totally geodesic
faces, finitely many geodesic edges and no vertices (other than the
ideal vertices at infinity).  Convexity of $CH(X)$ implies that every
interior dihedral angle of the polyheldron
is less than $\pi$. The closure $\overline U$ of the complement $U$ of
$CH(X)$ in $\BH^3$ is just the complement in $\BH^3$ of the interior
of the polyhedron $CH(X)$.  Doubling $\overline U$ we obtain a
hyperbolic cone--manifold $D\overline U$ with empty boundary and cone
angles greater than $2\pi$.  As above, the doubled $D\overline U$ is
CAT(-1) under the induced path metric
and, again as above, we deduce that $\overline U$ itself is CAT(-1)
with totally geodesic boundary under the induced path metric.
This concludes the proof of Proposition \ref{convex-cat}.
\end{proof}

We conclude with the following slightly more general version of
Proposition \ref{convex-cat}:

\begin{kor}\label{cor:convex-cat}
Let $\{X_i\}$ be a sequence of closed subsets of $\D_\infty\BH^3$ and
assume that for every $x \in \BH^3$ there is some $\epsilon_x>0$ such
that the ball $B_{\BH^3}(x,\epsilon_x)$ intersects at most one of the
convex hulls $CH(X_i)$.  Let $U$ be a connected component of $\BH^3
\setminus \cup_i CH(X_i)$ and $\widetilde{U}$ its universal cover.
Then $\overline{\widetilde{U}}$, the completion of $\widetilde{U}$
with respect to the lifted path metric, is a CAT(-1) space.
\end{kor}

\begin{proof}
The assumption that each point in $x \in \BH^3$ is the center of some
ball which only intersects one of the convex hulls $CH(X_i)$ implies
that each point in $x \in \overline{\widetilde{U}}$ has a small
neighborhood isometric to the completion of the universal cover of the
complement of $CH(X_i)$.  If $X_i$ consists of only two points, then
it follows from Soma \cite{Soma-shrink} that $x$ has a small CAT(-1)
neighborhood in $\overline{\widetilde{U}}$.  If $X_i$ has at least
three points then we obtain the same consequence from Proposition
\ref{convex-cat}.  Using this local description of the completion of
the universal cover, it is easy to see that every curve in
$\overline{\widetilde{U}}$ can be homotoped to the a curve in
$\widetilde U$ and hence that $\overline{\widetilde{U}}$ is simply
connected.  In other words, $\overline{\widetilde U}$ is a simply
connected locally CAT(-1) space and hence is CAT(-1) by the
Hadamard--Cartan theorem.
\end{proof}

\subsection{Proof of Theorem \ref{yair}}
We prove now Theorem \ref{yair}. 

\begin{named}{Theorem \ref{yair}}
Let $M$ be a hyperbolic 3--manifold.  If the manifold $M$ has at least
two convex cocompact ends, then $M$ is not the geometric limit of any
sequence of hyperbolic knot complements in $\BS^3$.
\end{named}

\begin{proof}
Let $M$ be as in the statement of the theorem and assume that there is
a sequence of hyperbolic knot complements $\{M_{K_i}\}$ converging
geometrically to $M$.  Let $\CN(CC(M))$ be a regular neighborhood of
$CC(M)$ with smooth strictly convex boundary and let $\Sigma_1$ and
$\Sigma_2$ be two compact components of $\D\CN(CC(M))$.  Fix now a
compact, connected 3--dimensional submanifold $W$ of $M$ whose boundary
contains both $\Sigma_1$ and $\Sigma_2$.  For some $d$ to be
determined below, let $W_d=\{x\in M\vert d_M(x,W)\le d\}$.  By
geometric convergence, for large $i$ we have better and better almost
isometric embeddings
$$f_i\co W_d \to M_{K_i}.$$ We may assume, passing to a subsequence,
that for all $i$ the surfaces $f_i(\Sigma_1)$ and $f_i(\Sigma_2)$ are
locally convex.

The two surfaces $f_i(\Sigma_1)$ and $f_i(\Sigma_2)$ are closed and
disjoint.  Since the complement of a knot in $\BS^3$ does not contain
non-separating closed surfaces, we obtain that $f_i(\Sigma_1)$ and
$f_i(\Sigma_2)$ separate the knot complement $M_{K_i}$ into three
pieces.  Let $V_i^0$ be the component containing $f_i(W)$.  If
$M_{K_i} \setminus V_i^0$ has an unbounded component, let $V_i^1$ be
this component; otherwise choose $V_i^1$ to be either remaining
component.  Set $V_i = V_i^0 \cup V_i^1$ and observe that it is
convex, and that its boundary is one of the two surfaces
$f_i(\Sigma_1), f_i(\Sigma_2)$.  Up to relabeling and passing to a
subsequence, we may assume $\D V_i = f_i(\Sigma_1)$ for all $i$.

It follows from the convexity of $V_i$ that its fundamental group
$\pi_1(V_i)$ injects into $\pi_1(M_{K_i})$.  Let $M_{V_i}$ be the
associated cover of $M_{K_i}$ and lift the inclusion $V_i
\hookrightarrow M_{K_i}$ to the inclusion $V_i \hookrightarrow
M_{V_i}$.  The convexity of $V_i$ implies that the convex core
$CC(M_{V_i})$ of $M_{V_i}$ is contained in $V_i$.  In particular, the
restriction of the covering $M_{V_i} \to M_{K_i}$ to $CC(M_{V_i})$ is
injective.  Before moving on we observe:
\begin{enumerate}
\item The submanifolds $CC(M_{V_i})$ and $V_i$ of $M_{K_i}$ are
isotopic.
\item The boundary of $CC(M_{V_i})$ is a closed connected surface
$S_i$.  The surfaces $S_i$ are uniformly bi-Lipschitz equivalent to
$\Sigma_1$.  In particular, there is some uniform $\delta$ with
$\inj(S_i) \ge \delta$ for all $i$.
\item The surface $\D CC(M_{V_i})$ has, for large $i$, a collar of at
least width $d$ in $M_{K_i}\setminus CC(M_{V_i})$,
since recall $f_i$ is an almost isometric embedding of $W_d$,
containing $W$. In particular, any essential arc $(\kappa, \D\kappa) \subset(M_{K_i}
\setminus CC(M_{V_i}), \D CC(M_{V_i}))$ has, for all sufficiently
large $i$ length at least $2d$.
\end{enumerate}

Let now $U_i' = M_{K_i} \setminus V_i$ and $U_i = M_{K_i} \setminus
CC(M_{V_i})$ be the complements of $V_i$ and $CC(M_{V_i})$ in
$M_{K_i}$ respectively.  It follows from (1) that $U_i$ and $U_i'$ are
isotopic and from the construction that $U_i' \subset U_i$.  Let
$\overline U_i = U_i \cup S_i$ be the closure of $U_i$ in $M_{K_i}$.
%% JSP:  Reword below so that it is more clear why cor 6.2 applies in
%% this case.
Now, $\overline U_i$ is the completion of the complement of a convex
set $CC(M_{V_i})$ in the hyperbolic manifold $M_{K_i}$.  Lift to the
universal $\BH^3$ of $M_{K_i}$.  Lifts of $CC(M_{V_i})$ are convex
hulls of sets $X_j$ in $\D \BH^3$.  By Corollary \ref{cor:convex-cat},
the completion of the universal cover of a component of $\BH^3
\setminus \cup_j CH(X_j)$ is a CAT(-1) space with geodesic boundary.
This is the universal cover of $\overline U_i$.  Hence with respect to
the path metric induced from the interior metric, $\overline U_i$ is a
locally CAT(-1) manifold with totally geodesic boundary.
%% Proposition \ref{convex-cat} and Corollary \ref{cor:convex-cat}
%% that, with respect to the path metric induced from the
%% interior metric, $\overline U_i$ is a locally CAT(-1) manifold with
%% totally geodesic boundary.

It is a well--known fact that a compact 3--manifold which admits a
hyperbolic metric with respect to which the boundary is totally
geodesic is irreducible, atoroidal and has incompressible and
acylindrical boundary.  The argument applies verbatim to locally
CAT(-1) metrics with again totally geodesic boundary.  We deduce hence
that $\overline U_i$ is irreducible, atoroidal and has incompressible
and acylindrical boundary.

Via the embedding $M_{K_i} \hookrightarrow \BS^3$, we consider
$\overline U_i$ as a submanifold of the 3--sphere and let $H_i = \BS^3
\setminus \overline U_i$ be its complement.  Since a component of $\D
H_i$ has at least genus $2$ we deduce that $H_i$ cannot be simply
connected.  In particular, the homomorphism
$$\pi_1(H_i) \to \pi_1(\overline U_i \cup H_i) = \pi_1(\BS^3) = 1$$
cannot be injective. We deduce as in \cite[Section
2]{Lackenby-attaching} that the acylindrical manifold $\overline U_i$
contains an immersed incompressible, boundary incompressible planar
surface
$$(X_i, \D X_i) \subset (\overline U_i, \D\overline U_i)$$ with
negative Euler characteristic $\chi(X_i) = 2-k_i$; here $k_i$ is the
number of boundary components of $X_i$.

We proceed now as for example in \cite{Soma-shrink} to obtain a
CAT(-1) metric with geodesic boundary on $X_i$.  Start with a
triangulation of $X_i$ with a single vertex at each boundary component
and no vertices in the interior.  Homotope the map $X_i \to \overline
U_i$ so that the boundary curves go to geodesics in $\D\overline U_i$.
Then, keeping the boundary fixed, homotope the remaining edges of the
1--skeleton of $X_i$ to geodesic arcs.  Finally homotope each of faces
of the triangulation to a ruled surface.

Pulling back the CAT(-1) metric of $\overline U_i$ we obtain a CAT(-1)
metric on $X_i$ with geodesic boundary.  The Gau\ss --Bonnet theorem
implies then that
\begin{equation}\label{eq:GB}
\vol(X_i) \le 2\pi(k_i-2) < 2\pi k_i
\end{equation}
Assume that $(\kappa,\D\kappa)\subset(X_i,\D X_i)$ is an essential
arc.  Since the surface $X_i$ is boundary incompressible we deduce
from (3) that the image of $\kappa$ has at least length $2d$.  On the
other hand, the map $X_i \to \overline U_i$ is 1--Lipschitz.  Hence
each of the boundary components of $X_i$ has an embedded collar of
width $d$ and all these collars are disjoint. At the same time, each
one of the boundary components is an essential curve in $S_i =
\D\overline U_i'$ and hence has at least length $\delta$ by (2). We
deduce that each one of the $k$ collars has at least volume $d\delta$
and that all these collars are disjoint. Thus
$$\vol(X_i)\ge k_id\delta.$$ 
In particular, if $d$ is sufficiently
large we obtain a contradiction to the area bound \eqref{eq:GB}.  This
concludes the proof of Theorem \ref{yair}.
\end{proof}

At this point we would like to add a few observations on the geometry
of complements of convex submanifolds in hyperbolic 3--manifolds.  For
example, remark that during the proof of Theorem \ref{yair} we have
essentially obtained the following result:

\begin{prop}\label{convex-embedding}
Let $M = \BH^3/\Gamma$ be a hyperbolic 3--manifold and $V \subset M$ a
3--dimensional submanifold with locally convex, compact boundary and
non-abelian fundamental group $\pi_1(V)$.  Consider $U = M \setminus
V$ and let $\overline U = U \cup \D V$ be its metric completion.  The
embedding $U \hookrightarrow M$ is isotopic (but perhaps not ambient
isotopic) to a second embedding $\phi\co U \hookrightarrow M$ with the
following properties:
\begin{enumerate}
\item $U \subset \phi(U)$ and moreover, if $\D V$ is smooth and none
of its components are totally geodesic, then $\phi(U) \setminus U$ is
homeomorphic to $\D\overline U \times \BR$.
\item $\phi$ extends continuously to a locally injective $\bar\phi\co
\overline U \to M$.  Moreover, unless some component of $V$ is a
regular neighborhood of a non-separating totally geodesic surface, the
map $\bar\phi$ is injective when restricted each connected component
of $\overline U$.
\item If we endow $\overline U$ with the unique interior distance such that the
%% (added the word "interior")
map $\bar\phi$ preserves the lengths of curves, then $\overline U$ is
a CAT(-1) space with totally geodesic boundary.\qed
\end{enumerate}
\end{prop}

As a consequence of Proposition \ref{convex-embedding} and Soma
\cite{Soma-shrink} we have:

\begin{kor}\label{convex-acylindrical}
Let $M = \BH^3/\Gamma$ be a hyperbolic 3--manifold and $V \subset M$ a
3--dimensional submanifold with locally convex, compact boundary.
Assume that no component of $V$ is simply connected and let $\overline
U$ be the metric completion of $U = M \setminus V$.  Then $\overline
U$ is irreducible, atoroidal and has incompressible and acylindrical
boundary.\qed
\end{kor}

%%%%%%%%%%%%%%%%%%%%%%%%%%%%%%%%%%%%%%%%%%%%%%%%%%%%%%%%%%%%%%%%%

\section{Various}\label{sec:various}
In this section we discuss some of the examples mentioned in the
introduction, we make some more or less interesting remarks and we ask
a few questions.
\medskip

We start constructing hyperbolic knots whose complements have very big
injectivity radius at some point.  In order to do so, it suffices to
remark that $\BH^3$ is a complete hyperbolic 3--manifold with trivial,
and hence finitely generated, fundamental group which is homeomorphic
to a ball and hence has a single end and embeds into $\BS^3$.  In
other words, it satisfies the conditions of Theorem \ref{thm:main} and
hence $\BH^3$ is a geometric limit of some sequence of hyperbolic knot
complements $\{M_{K_i}\}$.  Recall that this means that there are $p
\in \BH^3$ and $p_i \in M_{K_i}$ such that for all $r$ and $\epsilon$
and all sufficiently large $i$ the ball $B_p(r) \subset \BH^3$ with
center $p$ and radius $r$ can be $(1+\epsilon)$ isometrically embedded
into $M_{K_i}$ by a map which maps $p$ to $p_i$.  This implies that
for all sufficiently large $i$, the set of points in $M_{K_i}$ which
is at distance at most say $R/2$ of $p_i$ is simply connected.  In
other words, for all large $i$ we have $\inj(p_i, M_{K_i}) \ge R/2$.
Since $R$ was arbitrary we obtain:

\begin{named}{Corollary \ref{large-inj}}
For every $R>0$ there is a hyperbolic knot complement $M_K$ and $x \in
M_K$ with injectivity radius $\inj(x, M_K) > R$. \qed
\end{named}

The proof of Corollary \ref{large-inj} is somehow disappointing
because the involved knots are produced in a very very indirect
way.  We ask:

\begin{quest}
Give an explicit construction of knots as in Corollary
\ref{large-inj}.
\end{quest}

The same strategy used to prove Corollary \ref{large-inj} can be used
to show that there are hyperbolic knots whose complement has
arbitrarily large Heegaard genus, arbitrarily large volume,
arbitrarily many arbitrarily small eigenvalues of the Laplacian,
arbitrarily many arbitrarily short geodesics, contains surfaces with
arbitrarily small principal curvatures... It should be said that knots
with most of these properties were either known to exist (see for
instance \cite{Leininger}) or no one bothered to try to construct them
before.
\medskip

We show now how to prove that some other hyperbolic 3--manifolds not
covered by Theorem \ref{thm:main} are also geometric limits of knot
complements.

It is a by now almost folklore fact that every hyperbolic 3--manifold
$M$ homeomorphic to the trivial interval bundle $\Sigma_g \times \BR$
over a closed surface of genus $g$ and which has at least one
degenerate end is the geometric limit of a sequence of hyperbolic
3--manifolds $\{M_i\}$ where each of $M_i$ is homeomorphic to the
interior of a handlebody of genus $g$ (compare for instance with
\cite[Section 3]{biringer-souto}).  Each one of the manifolds $M_i$
satisfies the assumption of Theorem \ref{thm:main} and hence is a
geometric limit of knot complements.  It follows now from Lemma
\ref{lemma:silly}:

\begin{kor}\label{kor1}
Every hyperbolic 3--manifold $M$ homeomorphic to $\Sigma_g \times \BR$
and which has at least one degenerate end is the geometric limit of a
sequence of hyperbolic knot complements.\qed
\end{kor}

More generally, if $M$ is homeomorphic to the interior of a
compression body with exterior boundary of genus $g$ and such that
each of the interior ends is degenerate, then $M$ the geometric limit
of a sequence $\{M_i\}$ where each of the $M_i$ is homeomorphic to a
genus $g$ handlebody (compare again \cite[Section 3]{biringer-souto}).
Paying the price of having a large exterior boundary, we have
compression bodies with as many interior boundary components as we
wish.  In particular we deduce:

\begin{kor}\label{kor2}
For every $n$ and $g$ there is a hyperbolic 3--manifold $M$ with at
least $n$ ends which have neighborhoods homeomorphic to $\Sigma_g
\times \BR_+$ and which arise as geometric limits of knot
complements.\qed
\end{kor}

The statement of Corollary \ref{kor1} and Corollary \ref{kor2} become
more interesting when compared with Theorem \ref{yair} since the
latter shows that for instance not every hyperbolic manifold
homeomorphic to $\Sigma_g \times \BR$ is a geometric limit of knot
complements.
 
\begin{quest}\label{q1}
Assume that $M$ has finitely generated fundamental group, embeds into
the sphere and has more than one end. What are the possible hyperbolic
metrics on $M$ which make $M$ a geometric limit of knot complements?
Are there restrictions on the possible ending laminations of the
degenerate ends?
\end{quest}

\begin{quest}\label{q2}
Is there an irreducible and atoroidal 
%%(added irreducible and atoridal)
 3--manifold with finitely generated fundamental group
which embeds into $\BS^3$ and is not homeomorphic to any geometric
limit of knot complements?
\end{quest}

In relation with these two last questions, we would like to observe
that in the course of the proof of Theorem \ref{thm:main} we proved
something slightly stronger than what we state.  In fact, if $M$ is as
in the statement of the theorem and we fix an embedding $\overline M
\hookrightarrow \BS^3$ of the manifold compactification of $M$ into
the sphere, then we can find knots $K_i$ in the complement of
$\overline M$ in $\BS^3$ such that $M$ is a geometric limit of the
$M_{K_i}$.  In other words, if we identify $\overline M$ with a
standard compact core of $M$, then the initial embedding of $\overline
M$ into $\BS^3$ and the embedding obtained by composing the
%% JSP: reword
homeomorphism $\overline M \hookrightarrow M$ with the almost
isometric embedding $M \hookrightarrow M_{K_i}$ provided by geometric
convergence, followed by the standard embedding $M_{K_i}
\hookrightarrow \BS^3$, are isotopic.

It seems that whenever $M$ has at least two ends the situation
dramatically changes.  In particular, question \ref{q1} and question
\ref{q2} may actually turn to be questions on the possible
re-embeddings of submanifolds of the sphere.

Also related to question \ref{q1} and question \ref{q2} but in a
little different spirit we ask:

\begin{quest}
Is there a geometric limit of knot complements $M$ with finitely
generated $\pi_1(M)$ which has two geometrically finite geometric ends
contained in two different topological ends?
\end{quest}

For the sake of playfulness, or what is almost the same, to answer a
question of Yair Minsky, we discuss which hyperbolic 3--manifolds of
the form $\BH^3/\Gamma$, with $\Gamma \subset \PSL_2\BR$ a
torsion--free Fuchsian group, arise as geometric limits of knot
complements.  If the surface $\BH^2/\Gamma$ is compact, then
$\BH^3/\Gamma$ has two convex cocompact ends and hence is not a limit
of knot complements by Theorem \ref{yair}.  On the other hand, if the
surface $\BH^2/\Gamma$ is not compact then let $S_1 \subset S_2
\subset \dots$ be a nested sequence of compact connected
$\pi_1$--injective subsurfaces of $\BH^2/\Gamma$ with $\BH^2/\Gamma =
\cup S_i$.  Let $\Gamma_1 \subset \Gamma_2 \subset \dots$ be the
associated sequence of subgroups $\Gamma_i = \pi_1(S_i)$ of $\Gamma$.
Then the hyperbolic manifold $\BH^3/\Gamma$ is the geometric limit of
the sequence $\{\BH^3/\Gamma_i\}$.  On the other hand, each of the
manifolds $\BH^3/\Gamma_i$ is homeomorphic to a handlebody and hence
is a geometric limit of knot complements by Theorem \ref{thm:main}.
Combining these two observations we obtain:

\begin{kor}\label{fuchsian}
Let $\Gamma \subset \PSL_2\BR$ be a torsion free, but possibly
infinitely generated, Fuchsian group.  Then the hyperbolic manifold
$\BH^3/\Gamma$ is a geometric limit of knot complements if and only if
the surface $\BH^2/\Gamma$ is open.\qed
\end{kor}

At this point we would like to mention that the argument used to prove
Theorem \ref{thm:main} together with the remark after the proof of
Proposition \ref{prop:only-degenerate} imply that if $\BH^2/\Gamma$ is
a closed hyperbolic surface which admits an orientation preserving
involution without fixed points, then $\BH^3/\Gamma$ is the geometric
limit of a sequence of link complements where each link has two
components.

During the discussion of Corollary \ref{fuchsian} we observed that
Theorem \ref{thm:main} can also be used to prove that some hyperbolic
3--manifolds with infinitely generated fundamental group are geometric
limits of knot complements.  It may be that every one--ended
hyperbolic 3--manifold which embeds into the sphere is a geometric
limit of still one--ended hyperbolic 3--manifolds which embed into the
sphere and have finitely generated fundamental group.  Hence we ask:

\begin{quest}
Is it true that every one--ended hyperbolic 3--manifold which embedds
into $\BS^3$ is a limit of knot complements?  In other words, does
Theorem \ref{thm:main} holds for manifolds with infinitely generated
$\pi_1$?
\end{quest}

On the other hand, it could be that some hyperbolic 3--manifold $M$
which does not embed into the sphere is a geometric limit of
hyperbolic 3--manifolds which embed into the sphere.  This prompts the
following question:

\begin{quest}\label{question-last}
Is it true that every geometric limit of hyperbolic knot complements
embeds into $\BS^3$?
\end{quest}

The answer to question \ref{question-last} is most likely negative.

%%%%%%%%%%%%%%%%%%%%%%%%%%%%%%%%%%%%%%%%%%%%%%%%%%%%%%%%%%%%%%%%%
%%%%%%%%%%%%%%%%%%%%%%%%%%%%%%%%%%%%%%%%%%%%%%%%%%%%%%%%%%%%%%%%%
%% \bibliographystyle{hamsplain}
%% \bibliography{biblio.bib}

\providecommand{\bysame}{\leavevmode\hbox to3em{\hrulefill}\thinspace}
\providecommand{\href}[2]{#2}

\bigskip

\noindent Department of Mathematics, Brigham Young University
\newline \noindent
\texttt{jpurcell@math.byu.edu}

\bigskip

\noindent Department of Mathematics, University of Michigan
\newline \noindent
\texttt{jsouto@umich.edu}

\end{document}